\documentclass[a4paper,12pt]{article}
\usepackage{amsthm,amsmath,amssymb,latexsym,cite,verbatim,enumitem,appendix,mathrsfs}

\numberwithin{equation}{section}

\def\l{{\lambda}}
\def\R{\mathbb{R}}
\def\N{\mathbb{N}}
\def\Z{\mathbb{Z}}

\def\e{{\varepsilon}}

\def\p{\partial}

\def\ints{\int_{\R^n}}

\def\F{\mathcal{F}}

\def\til{~}

\newcommand{\n}[1]{\left\lVert#1\right\rVert}

\DeclareMathOperator*{\supp}{supp}
\DeclareMathOperator*{\sgn}{sgn}
\DeclareMathOperator*{\meas}{meas}

\newtheorem{thm}{Theorem}
\newtheorem{cor}{Corollary}
\newtheorem{lem}{Lemma}
\newtheorem{rem}{Remark}
\newtheorem{Def}{Definition}

\title{Blow-up and lifespan estimate for generalized Tricomi equations related to Glassey conjecture}
\author{
Ning-An Lai
	\footnote{
	Institute of Nonlinear Analysis and Department of Mathematics, Lishui University,
	No. 1 Xueyuan Rd., Lishui City 323000, Zhejiang, China.
	E-mail: ninganlai@lsu.edu.cn}
\quad
Nico Michele Schiavone
	\footnote{
	 Department of Mathematics,
	 Sapienza University of Rome,
	 P.le Aldo Moro 5, Roma 00185, Italy.
	E-mail: schiavone@mat.uniroma1.it
	}
}

\date{
\[
\begin{array}{ll}
\mbox{\footnotesize{\bf Keywords:}}
& \mbox{\footnotesize generalized Tricomi equations, semilinear, derivative nonlinearity, blow-up, lifespan}\\
\mbox{\footnotesize{\bf MSC2020:}}
& \mbox{\footnotesize primary 35L71, secondary 35B44 
}\\
\end{array}
\]
}

\begin{document}
\maketitle
\begin{abstract}

We study in this paper the small data Cauchy problem for the semilinear generalized Tricomi equations with a nonlinear term of derivative type
\[
u_{tt}-t^{2m}\Delta u=|u_t|^p
\]
for $m\ge0$. Blow-up result and lifespan estimate from above are established for $1<p\le 1+\frac{2}{(m+1)(n-1)-m}$. If $m=0$, our results coincide with those of the semilinear wave equation. The novelty consists in the construction of a new test function, by combining cut-off functions, the modified Bessel function and a harmonic function. Interestingly, if $n=2$ the blow-up power is independent of $m$. We also furnish a local existence result, which implies the optimality of lifespan estimate at least in the $1$-dimensional case.

\end{abstract}


\section{Introduction}
\par\quad
In this work, we consider the small data Cauchy problem for the semilinear generalized Tricomi equations with a power-nonlinearity of derivative type
\begin{equation}
\label{Tricomi}
\left\{
\begin{aligned}
& u_{tt} - t^{2m}\Delta u = |u_t|^p \quad\text{in $[0, T)\times\R^n$}, \\
& u(x,0)=\e f(x), \quad u_t(x,0)=\e g(x) \quad x\in\R^n,
\end{aligned}
\right.
\end{equation}
where $m\ge0$ is a real constant, $n\ge1$ is the dimension  and $\e>0$ is a \lq\lq small'' parameter. The initial data $f,g$ are compactly supported functions from the energy spaces
\begin{equation*}
	f\in H^1(\R^n),
	\quad
	g\in H^{1-\frac{1}{m+1}}(\R^n), 
\end{equation*}
and, without loss of generality, we may assume
\begin{equation}\label{supp}
\begin{aligned}
\supp ( f, g ) \subset \{x \in\R^n:|x|\le 1\}.
\end{aligned}
\end{equation}

Recently, the semilinear generalized Tricomi equations attract more and more attention, due to their physical background. In dimension $n=1$ and for $m=1/2$, the equation
\[
u_{tt}-tu_{xx}=0
\]
is the classical linear Tricomi equation, which was introduced by Tricomi \cite{Tri23} in the context of boundary value problems for partial differential equations of mixed type.
The Tricomi equation appears in gas dynamic problems, connected to gas flows with nearly sonic speed, describing the transition from subsonic flow ($t<0$) to supersonic flow ($t>0$), see Frankl \cite{Fra}.
We refer to Yagdjian \cite{Yag1,Yag2,Yag3,Yag4,Yag5} and references therein for more details.

For $k>0$ and $n\ge 1$, the operator
\begin{equation}\label{Gellerstedt}
\mathcal{T} = \partial_{t}^2-t^{k}\Delta
\end{equation}
appearing in \eqref{Tricomi} is also known as Gellerstedt operator.
 The first direction in the study of (generalized) Tricomi equations is to construct the explicit fundamental solution. Barros-Neto and Gelfand \cite{BG1, BG2, BG3} first established the fundamental solution for
\[
yu_{xx}+u_{yy}=0
\]
in the whole plane. For the Gellerstedt operator
\eqref{Gellerstedt} with $k\in\N$, Yagdjian \cite{Yag1} constructed a fundamental solution relative
to an  arbitrary point $(t_0,x_0) \in [0,+\infty)\times\R^{n}$ with support located in the \lq\lq forward cone'' $D(t_0,x_0)=\{(t,x)\in\R^{n+1} \colon (k+2)|x-x_0|\le 2(t-t_0)^{k/2+1}\}$.

Recently, the long time behavior of solutions for small data Cauchy problem to the semilinear generalized Tricomi equation
\begin{equation}
\label{Trico1}
\left\{
\begin{aligned}
& u_{tt} - t^k \Delta u = |u|^p, \quad\text{in $[0, T)\times\R^n$} \\
& u(x,0)=\e f(x), \quad u_t(x,0)=\e g(x) \quad x\in\R^n,
\end{aligned}
\right.
\end{equation}
%
%
comes into interest.
The main goal is to determine the \emph{critical power} $p_c(k, n)$ such that, if $1<p\le p_c(k, n)$ then the solution blows up in a finite time, while if $p>p_c(k, n)$ there is a unique global-in-time solution.
Yagdjian \cite{Yag2} obtained some partial results, in the sense that there was still a gap between the blow-up and global existence ranges. The critical power is finally established in a recent series of works by He, Witt and Yin \cite{He1, He2, He3, He4} (see also the Doctoral dissertation by He \cite{He0}). For $k\ge 1$,  $p_c(k, n)$ admits the following form:
\begin{itemize}
		\item if $n = 1$, then $\displaystyle p_c(k, 1)=1+\frac 4k$;
		\item if $n\ge 2$, then $p_c(k, n)$ is the positive root of the quadratic equation
		\begin{equation*}
			2 + \left[ n+1 -3 \left(1-\frac{2}{k+2} \right)\right] p - \left[ n-1 + \left(1-\frac{2}{k+2} \right)\right] p^2 =0.
			%
		\end{equation*}
\end{itemize}
Recently, Lin and Tu \cite{LinTu} studied the upper bound of lifespan estimate for \eqref{Trico1}, and Ikeda, Lin and Tu \cite{ILT} established the blow-up and upper bound of lifespan estimate for the weakly coupled system of the generalized Tricomi equations with multiple propagation speed. The critical power above should be compared with the corresponding one for the semilinear wave equation $u_{tt}-\Delta u = |u|^p$. Indeed, letting $k=0$ in the definition of $p_c(k,n)$, we infer $p_c(0,1)=+\infty$ and, for $n\ge2$, $p_c(0, n)$ becomes the
so-called Strauss exponent,
which is exactly the critical power for the small data Cauchy problem in \eqref{Trico1} with $k=0$ (see e.g. \cite{LST20} and references therein for background and results about this problem).
Finally, we refer also to Ruan, Witt and Yin \cite{RWY1, RWY2, RWY3, RWY4} for results about the local existence and local singularity structure of low regularity solutions for the equation $u_{tt} -t^k \Delta u = f(t,x,u)$.

In this paper, we consider the semilinear generalized Tricomi equations with power-nonlinearity of derivative type, focusing on blow-up result and lifespan estimate from above for the small data Cauchy problem.
Note that, setting $m=0$ in \eqref{Tricomi}, we come back to the semilinear wave equation
\begin{equation}
\label{wave}
u_{tt} - \Delta u = |u_t|^p.
%
\end{equation}
For this problem, Glassey \cite{Gla83} conjectured that the critical exponent is the power, now named after him, defined by
\begin{equation}\label{pG}
	p_G(n) :=
	\left\{
	\begin{aligned}
	&1+\frac{2}{n-1} &\text{if $n\ge2$,}
	\\
	&+\infty &\text{if $n=1$.}
	\end{aligned}
	\right.
\end{equation}
The research on this problem was initiated by John \cite{John81}, where more general equations in dimension $n=3$ are considered, proving the blow-up of solutions for $p=2$.
Then, the study of the blow-up was continued in the low dimensional case by Masuda \cite{Mas83}, Schaeffer \cite{Sch86}, John \cite{John85} and Agemi \cite{Age91}, whereas Rammaha \cite{Ram87} treated the high dimensional case $n\ge4$ under radial symmetric assumptions. Finally, Zhou \cite{Zhou01} proved the blow-up for $n\ge1$ and $1<p\le p_G(n)$, furnishing the upper bound for the lifespan of the solutions, namely
\begin{equation}\label{TG}
	T_\e \le
	\left\{
	\begin{aligned}
	& C \e^{-\frac{2(p-1)}{2-(n-1)(p-1)}} &&\text{if $1<p<p_G(n)$,}
	\\
	& \exp(C \e^{-(p-1)})
	&&\text{if $p=p_G(n)$,}
	\end{aligned}
	\right.
\end{equation}
for some positive constant $C$ independent of $\e$. We recall that the \emph{lifespan} $T_\e$ is defined as the maximal existence time of
the solution, depending on the parameter $\e$.
Regarding the global existence part, we refer to Sideris \cite{Sid83}, Hidano and Tsutaya \cite{HT95} and Tzvetkov \cite{Tzv98} for results in dimension $n=2,3$ and Hidano, Wang and Yokoyama \cite{HWY12} for the high dimensional cases $n\ge4$ under radially symmetric assumptions. For more details about the Glassey conjecture, one can read the references \cite{LT2} and \cite{Wang}.

The study of problem \eqref{Tricomi} under consideration generalize the Glassey conjecture. Therefore, it is interesting to find the critical exponent and lifespan estimate for \eqref{Tricomi}, which will coincide with the Glassey exponent \eqref{pG} and Zhou's lifespan estimate \eqref{TG} respectively for $m=0$.
The main tool here used is the test function method. In \cite{He2}, the blow-up result for \eqref{Trico1} is based on a test function given by the product of the harmonic function $\int_{S^{n-1}}e^{x\cdot\omega}d\omega$ and the solution of the ordinary differential equation
\[
\lambda''(t)-t^k\lambda'(t)=0.
\]
Inspired by the works \cite{ISWa} and \cite{LT1}, we construct a nonnegative test function composed by a cut-off function, the harmonic function $\int_{S^{n-1}}e^{x\cdot\omega}d\omega$ and the solution of the ODE \eqref{ODE} below. Since we consider the Tricomi-type equations with derivative nonlinear term, the first derivative with respect to time variable and a factor $t^{-2m}$ are included in the special test function.
Before proving the blow-up result, we give also a local existence result following the approach of Yagdjian \cite{Yag2}, from which we can deduce the optimality of the lifespan estimates at least for $n=1$.

When this work was almost finished, we found the paper \cite{LP} by Lucente and Palmieri, where they independently studied the same problem with a different approach. However, the result we are going to present here improves the blow-up range and lifespan estimates they found.

We cite also the very recent papers \cite{CLP} by Chen, Lucente and Palmieri and \cite{HH} by Hamouda and Hamza, where the blow-up phenomena for generalized Tricomi equations with combined linearity, i.e. $u_{tt}-t^{2m}\Delta u = |u_t|^p+|u|^q$, is independently studied exploiting the iteration argument. In particular, the work \cite{HH} confirms the blow-up result presented in this paper by giving an alternative proof. Conversely, we are confident that also our method can be adapted to study various blow-up problems involving generalized Tricomi equations, including the combined nonlinearity. This means that the test function method presented in this work and the iteration argument developed in \cite{CLP, HH} furnish two different approaches for the study of blow-up phenomena for Tricomi-related problems.

\section{Main Result}

Let us start stating the definition of energy solution for our problem \eqref{Tricomi}, similarly as in \cite{ISWa} and \cite{LT}.

\begin{Def}\label{def1}
We say that the function
\begin{equation*}
u\in C([0,T),H^1(\R^n))\cap C^1([0,T),H^{1-\frac{1}{m+1}}(\R^n)),
\quad
\text{with $u_t \in L_{\rm loc}^p((0,T) \times \R^n)$},
\end{equation*}
is a weak solution of \eqref{Tricomi} on $[0,T)$ if $u(0, x)=\e f(x)$ in $H^1(\R^n)$, $u_t(0, x)=\e g(x)$ in $H^{1-\frac{1}{m+1}}(\R^n)$ and
\begin{equation}\label{weaksol}
\begin{split}
& \e \int_{\R^n}g(x)\Psi(0, x)dx
+\int_0^T\int_{\R^n}|u_t|^p\Psi(t, x) \, dxdt \\
=& \int_0^T\int_{\R^n} - u_t(t, x)\Psi_t(t, x) \, dxdt
+ \int_0^T\int_{\R^n}
t^{2m}\nabla u(t, x)\cdot\nabla\Psi(t, x) \, dxdt ,
\end{split}
\end{equation}
for any $\Psi(t, x)\in C^1_0\left([0, T)\times \R^n\right) \cap C^{\infty}\left((0, T)\times \R^n\right)$.
\end{Def}

\begin{rem}
	The choice of the functional spaces $H^1(\R^n)$ and $H^{1-\frac{1}{m+1}}(\R^n)$ for the initial data $u(0, x)$ and $u_t(0, x)$ respectively are suggested by \cite{He0} and by Theorem\til\ref{thm:locex} below. Of course, if $m=0$ we have $H^0(\R^n) = L^2(\R^n)$.
\end{rem}


In the same spirit of \cite{LST20}, let us define the exponent
\begin{equation*}
	p_T(n,m):=
	\left\{
	\begin{aligned}
	& 1+\frac{2}{(m+1)(n-1)-m} &&\text{if $n\ge2$,} \\
	& +\infty &&\text{if $n=1$,}
	\end{aligned}
	\right.
\end{equation*}
as the root (when $n\ge2$) 
of the expression $\gamma_T(n,m;p)=0$, where
\begin{equation*}
	\gamma_T(n,m;p) := 2 - [(m+1)(n-1)-m](p-1),
\end{equation*}
and observe that $\gamma_T(n,m;p) > 0$ for $1<p<p_T(n,m)$.

We state now our main result for \eqref{Tricomi}.

\begin{thm}
\label{thmTricomi}
Let $n\ge 1$, $m\ge0$ and
$
1<p\le p_T(n,m).
$
Assume that $f\in H^1(\R^n)$, $g\in H^{1-\frac{1}{m+1}}(\R^n)$ satisfy the compact support assumption \eqref{supp} and that
\begin{equation}\label{conddata}
	 a(m) f + g,
	 \quad
	 a(m) := [2(m+1)]^{\frac{m}{m+1}} 
	 {\Gamma\left(
	 	\frac{1}{2}+\frac{m}{2(m+1)}
	 \right)} {\Gamma\left(
	 	\frac{1}{2}-\frac{m}{2(m+1)}
	 \right)}^{-1},
\end{equation}
is non-negative and not identically vanishing.
Suppose that $u$ is an energy solution of \eqref{Tricomi} with compact support in the \lq\lq cone''
\begin{equation}\label{suppcond}
	\supp u
	\in
	\left\{(t,x) \in [0,T) \times \R^n \colon |x| \le \gamma(t) := 1+\frac{t^{m+1}}{m+1}\right\}.
\end{equation}

Then, there exists a constant $\e_0=\e_0(f, g, m, n, p)>0$
such that the lifespan $T_\e$ has to satisfy
\begin{equation}
\label{lifespan2}
T_\e \le
 \left\{
 \begin{aligned}
 & C \e^{-\frac{2(p-1)}{\gamma_T(n,m;p)}}
 && \text{if $1<p<p_T(n,m)$,}
 \\
 &\exp\left(C\e^{-(p-1)}\right)
 && \text{if $p=p_T(n,m)$,}
 \end{aligned}
 \right.
\end{equation}
for $0<\e\le\e_0$ and some constant $C$ independent of $\e$.
\end{thm}

\begin{rem}
	For $m=0$ the exponent $p_T$ becomes the Glassey exponent \eqref{pG}, namely $p_T(n,0)=p_G(n)$, and the lifespan estimate \eqref{lifespan2} is exactly the same as \eqref{TG}.
\end{rem}

\begin{rem}
	It is interesting to see that, if $n=2$, then the blow-up power $p_T(2,m)=3$ and the subcritical lifespan estimate $T_\e \le C \e^{-\left( \frac{1}{p-1} - \frac{1}{2} \right)^{-1}}$ are independent of $m$, whereas the lifespan for the \lq\lq critical" power $p=p_T(n,m)$ is independent of $m$ for every dimension.
\end{rem}

\begin{rem}
We conjecture that $p_T(n,m)$ is indeed the critical exponent for problem \eqref{Tricomi} and the lifespan \eqref{lifespan2} are optimal. The next goal is to verify this conjecture considering the global-in-time existence for solutions to \eqref{Tricomi}.
\end{rem}


\section{Local existence result}\label{sec:locex}

Before to proceed with the demonstration of Theorem\til\ref{thmTricomi} in Section\til\ref{sec:blowup}, we firstly want to present in this section a local existence result.
As observed in \cite{LP}, it is possible to prove a local-in-time existence result for problem \eqref{Tricomi}, regardless the size of the Cauchy data, following the steps in Section 2.1 of \cite{DD01}. However, we believe that the following Theorem\til\ref{thm:locex} is interesting to justify the choice of the energy space for the solution and the initial data in Theorem\til\ref{thmTricomi}. In addition, we can verify the optimality of the lifespan estimate in the $1$-dimensional case.


Let us consider the integral equation
\begin{equation}\label{intTricomi}
\begin{split}
u(t,x) =&\, \e V_1(t,D_x) f(x) + \e V_2(t,D_x) g(x)
\\
&+ \int_0^t [V_2(t,D_x)V_1(s,D_x)-V_1(t,D_x)V_2(s,D_x)] |u_t(s,x)|^p ds
\end{split}
\end{equation}
where $\e>0$ is not necessarily small, $f \in H^1(\R^n)$, $g \in H^{1-\frac{1}{m+1}}(\R^n)$ and the Fourier multiplier $V_1(t,D_x)$ and $V_2(t,D_x)$ are defined below. As remarked in \cite{Yag2}, any classical or distributional solution to our problem \eqref{Tricomi} solves also the integral equation \eqref{intTricomi}. We have the following result.

\begin{thm}\label{thm:locex}
	Let $0\le m<2$, $p> \max\left\{2, 1+\frac{n}{2}\right\}$,
	$\sigma \in \left(\frac{n}{2} - \frac{m}{2(m+1)}, p-1-\frac{m}{2(m+1)} \right)$ and $f \in H^{\sigma+1}(\R^n)$, $g \in H^{\sigma+1-\frac{1}{m+1}}(\R^n)$. Then there exists a unique solution
	\begin{equation*}
	u \in
	C \left([0,T); \dot{H}^{\sigma+\frac{m}{2(m+1)}+1}(\R^n) \right) \cap C^1 \left([0,T); {H}^{\sigma+\frac{m}{2(m+1)}}(\R^n) \right)
	\end{equation*}
	to equation \eqref{intTricomi} for some $T>0$.
	
	Moreover, if $\e>0$ is small enough, then $T \gtrsim \e^{- \left(\frac{1}{p-1}+\frac{m}{2}\right)^{-1}} $.	
\end{thm}

As in Yagdjian \cite{Yag4} and Taniguchi and Tozaki \cite{TT80}, we introduce the differential operators $V_1(t,D_x)$ and $V_2(t,D_x)$ as follows. Set $z := 2i\phi(t)|\xi|$, $\phi(t) := \frac{t^{m+1}}{m+1}$ and $\mu := \frac{m}{2(m+1)}$. Then $V_1(t,D_x)$ and $V_2(t,D_x)$ are the Fourier multiplier
\begin{align*}
V_1(t,D_x)\psi = \F^{-1}[ V_1(t,|\xi|) \F\psi ],
\\
V_2(t,D_x)\psi = \F^{-1}[ V_2(t,|\xi|) \F\psi ],
\end{align*}
defined by the symbols
\begin{align*}
V_1(t,|\xi|) &:= e^{-z/2} \Phi( \mu, 2\mu; z),
\\
V_2(t,|\xi|) &:= t e^{-z/2} \Phi( 1-\mu, 2(1-\mu); z),
\end{align*}
where $\F,\F^{-1}$ are the Fourier transform and its inverse respectively, and $\Phi(a,c;z)$ is the confluent hypergeometric function. Recall that $\Phi(a,c;z)$ is an entire analytic function of $z$ such that
\begin{equation}\label{Phi0}
\Phi(a,c;z)=1+O(z)
\quad
\text{for $z\to 0$}
\end{equation}
%
and which satisfies the following differential relations (see e.g. \cite[Section 13.4]{AS}):
\begin{align}\label{derPhi1}
\frac{d^n}{d z^n} \Phi(a,c;z) &= \frac{(a)_n}{(c)_n} \Phi(a+n,c+n;z),
\\
\label{derPhi2}
\frac{d}{dz} \Phi(a,c;z) &= \frac{1-c}{z} \left[ \Phi(a,c;z) - \Phi(a,c-1;z) \right],
\end{align}
where $(x)_n=x(x+1)\cdots(x+n-1)$ is the Pochhammer’s symbol.
Moreover $\Phi(a,c;z)$ satisfies the estimate
\begin{equation}\label{estPhi}
|\Phi(a,c;2i\phi(t)|\xi|)| \le C_{a,c,m}
(\phi(t)|\xi|)^{\max\{a-c,-a\}}
\quad
\text{for $2\phi(t)|\xi| \ge 1$.}
\end{equation}

\begin{rem}
In the case of the wave equation, i.e. when $m=0$, the definitions of $V_1(t,D_x)$ and $V_2(t,D_x)$ should be understood taking the limit for $m\to0$ in their formulas. Indeed, using the identities 10.2.14, 13.6.3 and 13.6.14 in \cite{AS}, we get
\begin{align*}
	\lim_{m\to0} V_1(t,|\xi|) &= \lim_{\mu\to0} \Gamma\left(\mu+\frac{1}{2}\right) \left(\frac{z}{4}\right)^{1/2-\mu} I_{\mu-1/2}\left(\frac{z}{2}\right)
	= \cosh\left(\frac{z}{2}\right),
	\\
	\lim_{m\to0} V_2(t,|\xi|) &= \frac{2t}{z} \sinh\left(\frac{z}{2}\right),
\end{align*}
where $I_\nu(w)$ is the modified Bessel function of first kind. Thus for $m=0$ one recovers the well-known wave operators $V_1(t,D_x) = \cos(t\sqrt{-\Delta})$ and $V_2(t,D_x) = \frac{\sin(t\sqrt{-\Delta})}{\sqrt{-\Delta}}$.
\end{rem}

As Yagdjian observes, there are two different phase functions of two different waves hidden in $\Phi(a,c;z)$. More precisely, for $0<\arg z<\pi$, we can write (see also \cite{Inui62})
\begin{equation}\label{PhiH}
e^{-z/2} \Phi(a,c;z) =
\frac{\Gamma(c)}{\Gamma(a)} e^{z/2} H_+(a,c;z)
+
\frac{\Gamma(c)}{\Gamma(c-a)} e^{-z/2} H_-(a,c;z)
\end{equation}
where
\begin{align*}
H_+(a,c;z) &= \frac{e^{-i\pi(c-a)}}{e^{i\pi(c-a)}-e^{-i\pi(c-a)}} \frac{1}{\Gamma(c-a)} z^{a-c} \int_{\infty}^{(0+)} e^{-\omega} \omega^{c-a-1} \left(1- \frac{\omega}{z}\right)^{a-1} d\omega ,
\\
H_-(a,c;z) &= \frac{1}{e^{i\pi a}-e^{-i\pi a}} \frac{1}{\Gamma(a)} z^{-a} \int_{\infty}^{(0+)} e^{-\omega} \omega^{a-1} \left(1 + \frac{\omega}{z}\right)^{c-a-1} d\omega	.
\end{align*}
For $|z| \to \infty$ and $0<\arg z<\pi$, the following asymptotic estimates hold:
\begin{align*}
H_+(a,c;z) &\sim z^{a-c} \left[1 + \sum_{k=1}^{\infty} \frac{(c-a)_k (1-a)_k}{k!} z^{-k} \right],
\\
H_-(a,c;z) &\sim (e^{-i\pi}z)^{-a} \left[1 + \sum_{k=1}^{\infty} (-1)^k \frac{(a)_k (1+a-c)_k}{k!} z^{-k} \right].
\end{align*}
Combining the asymptotic estimates for $H_+(a,c;z)$ and $H_-(a,c;z)$ with their definitions, one can infer, for $2\phi(t)|\xi|\ge1$, the relations
\begin{align}
\label{estH+}
|\partial_t^k \partial_\xi^\beta H_+(a,c;2i\phi(t)|\xi|) |
&
\le C_{a,c,m,k,\beta} (\phi(t)|\xi|)^{a-c} \langle \xi \rangle^{\frac{k}{m+1}-|\beta|},
\\
\label{estH-}
|\partial_t^k \partial_\xi^\beta H_-(a,c;2i\phi(t)|\xi|) |
&
\le C_{a,c,m,k,\beta} (\phi(t)|\xi|)^{-a} \langle \xi \rangle^{\frac{k}{m+1}-|\beta|},
\end{align}
where $\langle \xi \rangle = (1+|\xi|^2)^{1/2}$ are the Japanese brackets.

Finally, let us introduce for simplicity of notation the operators
\begin{align*}
W_1(s,t,D_x) :=&\, V_1(t,D_x)V_2(s,D_x),
\\
W_2(s,t,D_x) :=&\, V_2(t,D_x)V_1(s,D_x),
\end{align*}
whose symbols, if we set $z:=2i\phi(t)|\xi|$ and $\zeta:= 2i\phi(s)|\xi|$, are given by
\begin{align*}
W_1(s,t,|\xi|) &= s e^{-(z+\zeta)/2}
\Phi(\mu,2\mu;z) \Phi(1-\mu,2(1-\mu);\zeta),
\\
W_2(s,t,|\xi|) &= t e^{-(z+\zeta)/2}
\Phi(\mu,2\mu;\zeta) \Phi(1-\mu,2(1-\mu);z).
\end{align*}

	The key estimates employed in the proof of Theorem\til\ref{thm:locex} are given in Corollary\til\ref{cor:WHest}, which comes straightforwardly from Theorem\til\ref{thm:estW}, and in Lemma\til\ref{lem:Hestimates}.
	The estimates in the following Theorem\til\ref{thm:estW}, which are of independent interest, are obtained adapting the argument exploited by Yagdjian \cite{Yag2} and Reissig \cite{Rei97a} for the case of the operators $V_1(t,D_x)$ and $V_2(t,D_x)$. In order to not weigh down the exposition, we postpone the proof of this theorem in Appendix\til\ref{app:YRest}.

\begin{thm}\label{thm:estW}
	Let $n\ge1$, $m \ge 0$, $\mu:=\frac{m}{2(m+1)}$ and $\psi \in C^\infty_0(\R^n)$.
	Then the following $L^q-L^{q'}$ estimates on the conjugate line, i.e. for $\frac{1}{q}+\frac{1}{q'}=1$, hold
	for $0<s\le t$ and for all admissible $q \in (1,2]$:
	\begin{enumerate}[label=(\roman*)]
		
		\item
		if $n \left(\frac{1}{q}-\frac{1}{q'}\right)
		- 1
		\le \sigma
		\le
		-\mu
		+
		n \left(\frac{1}{q}-\frac{1}{q'}\right)$,
		then
		\begin{equation*}
		\n{(\sqrt{-\Delta})^{-\sigma} W_1(s,t,D_x) \psi}_{L^{q'}}
		\\
		\lesssim
		(t/s)^{-m/2}
		s^{1+\left[\sigma-n\left(\frac{1}{q}-\frac{1}{q'}\right) \right](m+1)}
		\n{\psi}_{L^q};
		\end{equation*}
		
		\item
		if $n \left(\frac{1}{q}-\frac{1}{q'}\right)
		-1
		\le \sigma
		\le
		-1+\mu+
		n \left(\frac{1}{q}-\frac{1}{q'}\right)$,
		then
		\begin{equation*}
		\n{(\sqrt{-\Delta})^{-\sigma} W_2(s,t,D_x) \psi}_{L^{q'}}
		\\
		\lesssim
		(t/s)^{-m/2}
		s^{1+\left[\sigma-n\left(\frac{1}{q}-\frac{1}{q'}\right) \right](m+1)}
		\n{\psi}_{L^q};
		\end{equation*}
		
		\item
		if $n \left(\frac{1}{q}-\frac{1}{q'}\right)
		\le \sigma
		\le
		1 - \mu + n \left(\frac{1}{q}-\frac{1}{q'}\right)$,
		then
		\begin{equation*}
		\n{(\sqrt{-\Delta})^{-\sigma} \partial_t W_1(s,t,D_x) \psi}_{L^{q'}}
		\\
		\lesssim
		(t/s)^{m/2} s^{\left[\sigma-n \left( \frac{1}{q} - \frac{1}{q'} \right) \right](m+1)}
		\n{\psi}_{L^q};
		\end{equation*}
		
		\item
		if $\sigma
		=
		n \left(\frac{1}{q}-\frac{1}{q'}\right)$,
		then
		\begin{equation*}
		\n{(\sqrt{-\Delta})^{-\sigma} \partial_t W_2(s,t,D_x) \psi}_{L^{q'}}
		\\
		\lesssim
		(t/s)^{m/2}
		\n{\psi}_{L^q}.
		\end{equation*}
		
	\end{enumerate}
\end{thm}

\begin{rem}
	As in Yagdjian \cite{Yag2}, it is easy to obtain similar estimates for the (homogeneous) Besov spaces and then for the Sobolev-Slobodeckij spaces.
\end{rem}

In the previous theorem, choosing $q=q'=2$, $\sigma=-1$ for $W_1(s,t,D_x)$ and $W_2(s,t,D_x)$ and $\sigma=0$ in the case of their derivatives, we immediately get the following corollary.

\begin{cor}\label{cor:WHest}
	The following estimates hold
	\begin{align*}
	\n{W_j(s,t,D_x) \psi}_{\dot{H}^{\gamma+1}}
	&\lesssim (t s)^{-m/2} \n{\psi}_{\dot{H}^\gamma},
	\\
	\n{\partial_t W_j(s,t,D_x) \psi}_{\dot{H}^\gamma}
	&\lesssim (t/s)^{m/2} \n{\psi}_{\dot{H}^\gamma},
	\\
	\n{\partial_t W_j(s,t,D_x) \psi}_{H^\gamma}
	&\lesssim (t/s)^{m/2} \n{\psi}_{H^\gamma},
	\end{align*}
	for $n\ge1$, $m \ge0$, $\gamma\in\R$ and $j\in\{1,2\}$.
\end{cor}


Now we furnish estimates in the energy space $\dot{H}^\gamma(\R^n)$ and ${H}^\gamma(\R^n)$ also for $V_1(t,D_x), V_2(t,D_x)$ and their derivatives respect to time.

\begin{lem}\label{lem:Hestimates}
	Let $\gamma\in\R$, $m\ge0$ and $\mu := \frac{m}{2(m+1)}$. The following estimates hold:
	\begin{align*}
	\n{V_1(t,D_x) \psi}_{\dot{H}^{\gamma-\sigma}} &\lesssim t^{\sigma(m+1)} \n{\psi}_{\dot{H}^\gamma}
	&&\quad\text{for $-\mu \le \sigma \le 0$;}
	\\
	\n{V_2(t,D_x) \psi}_{\dot{H}^{\gamma-\sigma}} &\lesssim t^{\sigma(m+1)+1} \n{\psi}_{\dot{H}^\gamma}
	&&\quad\text{for $-1+\mu \le \sigma \le 0$;}
	\\
	\n{\partial_t V_1(t,D_x) \psi}_{{H}^{\gamma-\sigma}} &\lesssim t^{\sigma(m+1)-1} \n{\psi}_{{H}^\gamma}
	&&\quad\text{for $1-\mu \le \sigma \le 1$;}
	\\
	\n{\partial_t V_2(t,D_x) \psi}_{{H}^{\gamma-\sigma}} &\lesssim \langle t \rangle^{\sigma(m+1)} \n{\psi}_{{H}^\gamma}
	&&\quad\text{for $\mu \le \sigma $.}
	\end{align*}
\end{lem}

\begin{proof}
	By estimates \eqref{estPhi}, for the range of $\sigma$ in the hypothesis we have that
	\begin{align*}
	||\xi|^{-\sigma}
	V_1(t,|\xi|) |
	&\lesssim
	\left\{
	\begin{aligned}
	& |\xi|^{-\sigma} (\phi(t)|\xi|)^{-\mu} &\text{if $\phi(t)|\xi|\ge1$,}
	\\
	& |\xi|^{-\sigma} &\text{if $\phi(t)|\xi|\le1$,}
	\end{aligned}
	\right.
	\\
	&\le
	t^{\sigma(m+1)},
	\\
	||\xi|^{-\sigma}
	V_2(t,|\xi|) |
	&\lesssim
	\left\{
	\begin{aligned}
	& t |\xi|^{-\sigma} (\phi(t)|\xi|)^{\mu-1} &\text{if $\phi(t)|\xi|\ge1$,}
	\\
	& t |\xi|^{-\sigma} &\text{if $\phi(t)|\xi|\le1$,}
	\end{aligned}
	\right.
	\\
	&\le
	t^{\sigma(m+1)+1},
	\\
	|\langle \xi \rangle^{-\sigma}
	\partial_t V_1(t,|\xi|) |
	&\lesssim
	\left\{
	\begin{aligned}
	& t^m |\xi|^{1-\sigma} (\phi(t)|\xi|)^{-\mu} &\text{if $\phi(t)|\xi|\ge1$,}
	\\
	& t^m |\xi|^{1-\sigma} &\text{if $\phi(t)|\xi|\le1$,}
	\end{aligned}
	\right.
	\\
	&\le
	t^{\sigma(m+1)-1},
	\\
	| \langle \xi \rangle^{-\sigma}
	\partial_t V_2(t,|\xi|) |
	&\lesssim
	\left\{
	\begin{aligned}
	& \langle \xi \rangle^{-\sigma} (\phi(t)|\xi|)^{\mu} &\text{if $\phi(t)|\xi|\ge1$,}
	\\
	& \langle \xi \rangle^{-\sigma} [1+ \phi(t)|\xi| ] &\text{if $\phi(t)|\xi|\le1$,}
	\end{aligned}
	\right.
	\\
	&\le
	\langle \xi \rangle^{-\sigma}
	\langle \phi(t)|\xi| \rangle^\mu
	\\
	&\le
	\langle \phi(t) \rangle^\mu \langle \xi \rangle^{\mu-\sigma}
	\\
	&\lesssim
	\langle t \rangle^{\sigma(m+1)}.
	\end{align*}
	Consequently
	\begin{align*}
	\n{V_1(t,D_x) \psi}_{\dot{H}^{\gamma-\sigma}}
	=&\,
	\n{ |\xi|^{\gamma-\sigma} V_1(t,|\xi|) \widehat{\psi}}_{L^2}
	\\
	\le&\,
	\n{ |\xi|^{-\sigma} V_1(t,|\xi|) }_{L^\infty}
	\n{|\xi|^\gamma \widehat{\psi}}_{L^2}
	\\
	\le&\,
	t^{\sigma(m+1)} \n{\psi}_{\dot{H}^\gamma},
	\end{align*}	
	and similarly we can obtain the other estimates.
\end{proof}

\begin{rem}
	The previous lemma should be compared with  Lemma 3.2 in \cite{RWY2}, where similar estimates are obtained under the restriction $0<t\le T$, for some fixed positive constant $T$.
\end{rem}

Finally, let us also recall the following useful relations that come from an application of Theorems 4.6.4/2 \& 5.4.3/1 in \cite{RS11}.

\begin{lem}\label{lem:RS}
	The following estimates hold:
	\begin{enumerate}[label=(\roman*)]
		\item if $\gamma>0$ and $u,v \in H^\gamma(\R^n) \cap L^\infty(\R^n)$, then
		\begin{equation*}
		\n{u v}_{H^\gamma} \lesssim
		\n{u}_{L^\infty}\n{v}_{H^\gamma}
		+
		\n{u}_{H^\gamma} \n{v}_{L^\infty};
		\end{equation*}
		
		\item if $p>1$, $\gamma\in \left(\frac{n}{2},p\right)$ and $u \in H^\gamma(\R^n)$, then
		\begin{equation*}
		\n{|u|^p}_{H^\gamma} \lesssim
		\n{u}_{H^\gamma} \n{u}^{p-1}_{L^\infty}.
		\end{equation*}
	\end{enumerate}
\end{lem}

We can now start the proof of the local existence result.

\begin{proof}[Proof of Theorem\til\ref{thm:locex}]
	Let us consider the map
	\begin{align*}
	\Psi[v](t,x) =&\,
	\e V_1(t,D_x)  f(x) +  \e V_2(t,D_x)  g(x)
	\\
	&+
	\int_0^t [V_2(t,D_x)V_1(s,D_x) - V_1(t,D_x)V_2(s,D_x)] |v_t(s,x)|^p ds
	\end{align*}
	and the complete metric space
	\begin{equation*}
	X(a,T) := \left\{
	v \in C([0,T); \dot{H}^{\gamma+1}(\R^n)) \cap C^1([0,T); {H}^{\gamma}(\R^n))
	\colon
	\n{v}_X \le a
	\right\}
	\end{equation*}
	for some $a,T>0$ to be chosen later, where $\gamma := \sigma+ \mu$, $\mu:=\frac{m}{2(m+1)}$ and
	\begin{equation*}
	\n{v}_X :=
	\sup_{0\le t \le T}
	\left[ t^{m/2} \n{v}_{\dot{H}^{\gamma+1}} + \langle t \rangle^{-m/2} \n{v_t}_{{H}^\gamma} \right] .
	\end{equation*}
	Note that, since the operators $V_1(t,D_x)$ and $V_2(t,D_x)$ commutes, we have
	\begin{align*}
	\partial_t \Psi[v](t,x) =&\,
	\e \partial_t V_1(t,D_x)  f(x) +  \e \partial_t V_2(t,D_x)  g(x)
	\\
	&+
	\int_0^t [ \partial_t V_2(t,D_x)V_1(s,D_x) - \partial_t V_1(t,D_x)V_2(s,D_x)] |v_t(s,x)|^p ds.
	\end{align*}
	We want to show that $\Psi$ is a contraction mapping on $X(a,T)$.
	
	By Lemma \ref{lem:Hestimates} and the immersion $H^s(\R^n) \hookrightarrow \dot{H}^s(\R^n)$ for $s>0$, we get
	\begin{align*}
	\n{ V_1(t,D_x)  f }_{\dot{H}^{\gamma+1}}
	&
	\lesssim
	t^{-m/2}
	\n{  f }_{{H}^{\gamma-\mu+1}}
	\\
	\n{ V_2(t,D_x) g }_{\dot{H}^{\gamma+1}}
	&\lesssim
	t^{-m/2}
	\n{  g }_{{H}^{\gamma+\mu}}
	\\
	\n{ \partial_t V_1(t,D_x)  f }_{{H}^{\gamma}}
	&\lesssim t^{m/2}
	\n{  f }_{{H}^{\gamma-\mu+1}}
	\\
	\n{ \partial_t V_2(t,D_x)  g }_{{H}^{\gamma}}
	&\lesssim \langle t \rangle^{m/2}
	\n{  g }_{{H}^{\gamma+\mu}}.
	\end{align*}
	Moreover by Corollary\til\ref{cor:WHest} we infer
	\begin{align*}
	\n{V_2(t,D_x)V_1(s,D_x) |v_t(s,x)|^p}_{\dot{H}^{\gamma+1}}
	&\lesssim
	%
	(st)^{-m/2} \n{v_t(s,x)}^p_{{H}^\gamma},
	\\
	\n{V_1(t,D_x)V_2(s,D_x) |v_t(s,x)|^p}_{\dot{H}^{\gamma+1}}
	&\lesssim
	%
	(st)^{-m/2}  \n{v_t(s,x)}^p_{{H}^\gamma},
	%
	\\
	\n{\partial_t V_2(t,D_x)V_1(s,D_x) |v_t(s,x)|^p}_{{H}^\gamma}
	&\lesssim
	%
	(s/t)^{-m/2}  \n{v_t(s,x)}^p_{{H}^\gamma},
	\\
	\n{\partial_t V_1(t,D_x)V_2(s,D_x) |v_t(s,x)|^p}_{{H}^\gamma}
	&\lesssim
	%
	(s/t)^{-m/2}  \n{v_t(s,x)}^p_{{H}^\gamma},
	\end{align*}
	where we used the estimates
	\begin{equation*}
		\n{|v_t(s,x)|^p}_{\dot{H}^\gamma} \le \n{|v_t(s,x)|^p}_{{H}^\gamma} \lesssim \n{v_t(s,x)}^p_{{H}^\gamma},
	\end{equation*}
	which
	come from Lemma\til\ref{lem:RS} and Sobolev embeddings.
	
	From these estimates and from the fact that $t\langle t \rangle^{-1} <1$ for any $t>0$, we obtain
	\begin{multline*}
	t^{m/2}
	\n{\Psi[v](t,\cdot)}_{\dot{H}^{\gamma+1}}
	+
	\langle t \rangle^{-m/2}
	\n{\partial_t \Psi[v](t,\cdot)}_{{H}^{\gamma}}
	\\
	\lesssim
	\e \left[ \n{ f}_{{H}^{\gamma-\mu+1}} + \n{ g}_{{H}^{\gamma+\mu}} \right]
	+
	\int_0^t s^{-m/2} \n{v_t(s,\cdot)}^p_{{H}^\gamma} ds
	\end{multline*}
	and hence, since $m<2$, we get
	\begin{equation*}
	\n{\Psi[v]}_{X} \le
	C_0 \e \left[ \n{ f}_{{H}^{\gamma-\mu+1}} + \n{ g}_{{H}^{\gamma+\mu}} \right] + C_0
	T^{1-\frac{m}{2}}
	\langle T \rangle^{\frac{m}{2}p}\n{v}^p_X
	\end{equation*}
	for some constant $C_0>0$ independent of $\e$.
	Choosing $a$ sufficiently large and $T$ sufficiently small, i.e. $a \ge 2C_0 \e \left[ \n{ f}_{{H}^{\gamma-\mu+1}} + \n{ g}_{{H}^{\gamma+\mu}} \right] $ and
	$T^{1-\frac{m}{2}}
	\langle T \rangle^{\frac{m}{2}p} \le (2C_0a^{p-1})^{-1}$,
	we infer that $\Psi[v] \in X(a,T)$.
	
	Now we show that $\Psi$ is a contraction. Fixed $v, \widetilde{v} \in X(a,T)$, we have similarly as above
	\begin{multline}\label{contr}
	t^{m/2}
	\n{\Psi[v](t,\cdot)
		- \Psi[\widetilde{v}](t,\cdot)}_{\dot{H}^{\gamma+1}}
	+
	\langle t \rangle^{-m/2}
	\n{\partial_t \Psi[v](t,\cdot)
		-
		\partial_t \Psi[\widetilde{v}](t,\cdot)}_{{H}^{\gamma}}
	\\
	\lesssim
	\int_0^t s^{-m/2} \n{|v_t(s,\cdot)|^p - |\widetilde{v}_t(s,\cdot)|^p}_{{H}^\gamma} ds.
	\end{multline}
	Since we can write
	\begin{equation*}
	|v_t|^p-|\widetilde{v_t}|^p =
	2^{-p} p \int_{-1}^{1} (v_t-\widetilde{v_t})(v_t+\widetilde{v_t}+ \lambda(v_t-\widetilde{v_t}))|v_t+\widetilde{v_t}+ \lambda(v_t-\widetilde{v_t})|^{p-2} d\lambda
	\end{equation*}
	and recalling that $p>2$ and $\gamma\in(n/2,p-1)$, an application of Lemma\til\ref{lem:RS} combined with Sobolev embeddings give us
	\begin{align*}
	\n{|v_t|^p-|\widetilde{v_t}|^p}_{H^\gamma} \lesssim&\,
	\n{v_t-\widetilde{v_t}}_{L^\infty}
	\n{(|v_t|+|\widetilde{v_t}|)^{p-1}}_{H^\gamma}
	\\
	&+ \n{v_t-\widetilde{v_t}}_{H^\gamma}
	\n{(|v_t|+|\widetilde{v_t}|)^{p-1}}_{L^\infty}
	\\
	\lesssim&\,
	\n{v_t-\widetilde{v_t}}_{L^\infty}
	\left(
	\n{v_t}^{p-1}_{H^\gamma}
	+
	\n{\widetilde{v_t}}^{p-1}_{H^\gamma}
	\right)
	\\
	&+
	\n{v_t-\widetilde{v_t}}_{H^\gamma}
	\left(
	\n{v_t}^{p-1}_{L^\infty}
	+
	\n{\widetilde{v_t}}^{p-1}_{L^\infty}
	\right)
	\\
	\lesssim&\,
	\n{v_t-\widetilde{v_t}}_{H^\gamma}
	\left(
	\n{v_t}^{p-1}_{H^\gamma}
	+
	\n{\widetilde{v_t}}^{p-1}_{H^\gamma}
	\right).
	\end{align*}
	Inserting this inequality into \eqref{contr} we get
	\begin{equation*}
	\n{\Psi[v] - \Psi[\widetilde{v}]}_X
	\le C_1
	T^{1-\frac{m}{2}}
	\langle T \rangle^{\frac{m}{2}p}
	a^{p-1} \n{v-\widetilde{v}}_X
	\end{equation*}
	for some $C_1>0$, and so $\Psi$ is a contraction for $T^{1-\frac{m}{2}}
	\langle T \rangle^{\frac{m}{2}p} \le (C_1 a^{p-1})^{-1}$.
	By Banach fixed point theorem we conclude that there exists a unique $v \in X(a,T)$ such that $\Psi[v]=v$.
	
	As a by-product of the computations, from the conditions on $T$ and $a$ we can choose
	the existence time such that
	$T^{1-\frac{m}{2}}
	\langle T \rangle^{\frac{m}{2}p} = C \e^{-(p-1)}$
	for some $C>0$ independent of $\e$, hence $T \gtrsim \e^{- \left[\frac{1}{p-1}+\frac{m}{2}\right]^{-1}} $ for $\e$ small enough.	%
\end{proof}

\section{Blow-up via a test function method}\label{sec:blowup}


We come now to the proof of Theorem \ref{thmTricomi}, which heavily relies on a special test function, closely related to a time dependent function satisfying the following ordinary differential equation:
\begin{equation}\label{ODE}
	\l''(t) - 2m t^{-1} \l'(t) - t^{2m} \l(t) = 0,
\end{equation}
where $t>0$ and $m\in\R$.

\begin{lem}\label{lem1}
	The fundamental solutions $\l_-$, $\l_+$ of \eqref{ODE} are the functions defined by:
	\begin{itemize}
		\item if $m = -1$:
		\begin{equation*}
			\l_-(t) = t^{-\frac{1+\sqrt{5}}{2}},
			\qquad
			\l_+(t) = t^{-\frac{1-\sqrt{5}}{2}};
		\end{equation*}
		
		\item if $m \neq -1$:
		\begin{equation*}
		\l_-(t) = t^{m+1/2} K_{\frac{1}{2}+\frac{m}{2(m+1)}} \left( \frac{t^{m+1}}{|m+1|} \right),
		\quad
		\l_+(t) = t^{m+1/2} I_{\frac{1}{2}+\frac{m}{2(m+1)}} \left( \frac{t^{m+1}}{|m+1|} \right),
		\end{equation*}
		where $I_{\nu}(z), K_{\nu}(z)$ are the modified Bessel functions of the first and second kind respectively.
	\end{itemize}
\end{lem}

\begin{proof}
	The result trivially follows from straightforward computations based on formulas for Bessel functions collected in Appendix \ref{appA}. Instead in Appendix \ref{appB} we show a way to reach the expression of the solutions for $m\neq-1$.
	
	If $m = -1$, it is immediate to check that $\l_-$ and $\l_+$ are two independent solutions of \eqref{ODE}.
	Suppose now $m\neq-1$ and set $z=t^{m+1}/|m+1|$ and $s= \sgn\{m+1\}$ for simplicity. From \eqref{ch2A7} and \eqref{ch2A2}, we get
	\begin{equation*}
	\begin{split}
	\l_-'(t) =&
	\left(m+\frac{1}{2}\right) t^{m-1/2} K_{\frac{m+1/2}{m+1}}(z)
	+ s t^{2m+1/2} K'_{\frac{m+1/2}{m+1}}(z)
	\\
	=& \left(m+\frac{1}{2}\right) t^{m-1/2} K_{\frac{m+1/2}{m+1}}(z)
	- s t^{2m+1/2}\\
	&\times\left[ K_{-\frac{1}{2(m+1)}}(z) + s \left(m+\frac{1}{2}\right) t^{-m-1} K_{\frac{m+1/2}{m+1}}(z) \right]	
	\\
	=& - s t^{2m+1/2} K_{-\frac{1}{2(m+1)}}(z)
	\\
	=& - s t^{2m+1/2} K_{\frac{1}{2(m+1)}}(z),
	\end{split}
	\end{equation*}
	and
	\begin{equation*}
	\begin{split}
	\l_-''(t) =&
	- s \left(2m+\frac{1}{2}\right) t^{2m-1/2} K_{\frac{1}{2(m+1)}}(z)
	- t^{3m+1/2} K'_{\frac{1}{2(m+1)}}(z)
	\\
	=&
	- s \left(2m+\frac{1}{2}\right) t^{2m-1/2} K_{\frac{1}{2(m+1)}}(z)\\
	&+ t^{3m+1/2} \left[ K_{-\frac{m+1/2}{m+1}}(z) + s \frac{t^{-m-1}}{2} K_{\frac{1}{2(m+1)}}(z) \right]
	\\
	=& t^{3m+1/2} K_{\frac{m+1/2}{m+1}}(z) - 2m s\, t^{2m-1/2} K_{\frac{1}{2(m+1)}}(z)
	\\
	=& t^{2m} \l_-(t) + 2m t^{-1} \l_-'(t).
	\end{split}
	\end{equation*}
	Analogously, using \eqref{ch2A6} and \eqref{ch2A7}, we obtain
	\begin{align*}
	\l_+'(t) &= s t^{2m+1/2} I_{-\frac{1}{2(m+1)}}(z),
	\\
	\l_+''(t) &=
	t^{3m+1/2} I_{\frac{m+1/2}{m+1}}(z)  +
	2ms \, t^{2m-1/2} I_{-\frac{1}{2(m+1)}}(z)
	\\
	&= t^{2m} \l_+(t) + 2m \, t^{-1} \l_+(t).
	\end{align*}
	Then, it is clear that $\l_-$ and $\l_+$ solve equation \eqref{ODE} and, from relation \eqref{ch2wr2}, we can check that the Wronskian $W(t)=\l_-(t) \l_+'(t) - \l_+(t) \l_-'(t)$ is
	\begin{equation*}
	W(t) =
	s t^{3m+1} [  I_{-\frac{1}{2(m+1)}} K_{-\frac{1}{2(m+1)}+1} + I_{-\frac{1}{2(m+1)}+1} K_{-\frac{1}{2(m+1)}}](z) = (m+1) t^{2m} > 0
	\end{equation*}
	for $t>0$, hence the two solutions are independent.
\end{proof}

\begin{lem}\label{lem2}
	Suppose $m>-1/2$. Define $\mu := \frac{m}{2(m+1)}$ and
	\begin{equation*}
		\l(t) := t^{m+1/2} K_{\mu+\frac{1}{2}} \left( \frac{t^{m+1}}{m+1} \right).
	\end{equation*}
	Then, $\l \in C^1([0,+\infty)) \cap C^\infty(0,+\infty)$  and satisfies the following properties:
	\begin{enumerate}[label=(\roman*)]
	\item
	$\displaystyle
	\l(t) > 0$,
	$\displaystyle
	\l'(t) < 0$,
	
	\item
	$\displaystyle
	\lim_{t\to 0^+} \l(t) =
		2^{\mu-\frac{1}{2}} (m+1)^{\mu+\frac{1}{2}} \Gamma\left(\mu+\frac{1}{2}\right) =: c_0(\mu)>0$,
	
	\item
	$\displaystyle
	\lim_{t\to 0^+} \frac{\l'(t)}{t^{2m}} =
		-c_0(-\mu)<0$,
	
	\item
	$\displaystyle
	\l(t)
	= \sqrt{\frac{(m+1)\pi}{2}} \, t^{m/2} \exp\left( -\frac{t^{m+1}}{m+1}\right)
	\times (1+O(t^{-(m+1)}))$,
	\quad for large $t>0$,
	
	\item
	$\displaystyle
	\l'(t)=
	- \sqrt{\frac{(m+1)\pi}{2}} \, t^{3m/2} \exp\left( -\frac{t^{m+1}}{m+1}\right)
	\times (1+O(t^{-(m+1)}))$,
	\quad for large $t>0$.
	\end{enumerate}

\end{lem}

\begin{proof}
	From \eqref{ch2A6} we know that $\l$ is smooth for $t>0$. Since $K_{\nu}(z)$ is real and positive for $\nu\in\R$ and $z>0$, also $\l$ is real and positive.
	Recall from the proof of Lemma \ref{lem1} that
	\begin{equation*}
		\l'(t) = - t^{2m+1/2} K_{-\mu+\frac{1}{2}} \left(\frac{t^{m+1}}{m+1}\right),
	\end{equation*}
	and hence $\l'$ is negative.
	From \eqref{limit0} we have $\l(t) \sim c_0(\mu)$ and $\l'(t) \sim -c_0(-\mu)t^{2m}$ for $t\to0^+$,
	so we prove (ii) and (iii). Finally, from \eqref{ch2A12} we obtain (iv) and (v).
\end{proof}


We can start now the proof of our main theorem.

\begin{proof}[Proof of Theorem\til\ref{thmTricomi}]
As in \cite{ISWa}, let $\eta(t)\in C^{\infty}([0, +\infty))$ satisfy
\[
\eta(r) :=
 \left\{
 \begin{aligned}
 &1
 && \text{for $r\le\tfrac12$},
 \\
 &\text{decreasing}
 && \text{for $\tfrac12<r<1$},
 \\
 &0
 && \text{for $r\ge 1$},
 \end{aligned}
 \right.
\]
and denote, for $M\in (1, T)$,
\[
\eta_M(t) := \eta\left(\frac tM \right),
\qquad
\eta^0(t,x) := \eta\left( \frac{|x|}{2} \left(1+\frac{t^{m+1}}{m+1}\right)^{-1} \right).
\]
We remark that one can assume $1<T \le T_\e$, since otherwise our result holds obviously by choosing $\e$ small enough.
The last ingredient other than $\l$, $\eta_M$ and $\eta^0$ to construct the test function is
\[
\phi(x) :=
\left\{
\begin{aligned}
&\int_{S^{n-1}}e^{x\cdot\omega}d\omega &\text{if $n\ge2$,}
\\
& e^x + e^{-x} &\text{if $n=1$,}
\end{aligned}
\right.
\]
which satisfies
\begin{equation}
\label{testpro}
\Delta\phi=\phi,
\quad
0<\phi(x)\le C_0 (1+|x|)^{-\frac{n-1}{2}}e^{|x|},
\end{equation}
for some $C_0>0$.
Finally, we introduce our test function as
\begin{equation}\label{testfunc}
	\begin{aligned}
	\Phi(t, x)
	:=& -t^{-2m}\partial_t\left(\eta_M^{2p'}(t)\l(t)\right) \phi(x) \eta^0(t,x)
	\\
	=& -t^{-2m}\left(\p_t\eta_M^{2p'}(t)\l(t) + \eta_M^{2p'}(t)\l'(t)\right) \phi(x) \eta^0(t,x)
	,
	\end{aligned}
\end{equation}
where $M\in (1, T)$ and $p'=p/(p-1)$ is the conjugate exponent of $p$.
It is straightforward to check that $\Phi(t,x) \in C^1_0([0,+\infty)\times\R^n) \cap C^\infty_0((0,+\infty)\times\R^n)$ if we set
\begin{equation*}
	\Phi(0,x) := \lim_{t \to 0^+} \Phi(t,x) = c_0(-\mu) \phi(x) \eta\left(\frac{|x|}{2}\right) \ge 0,
\end{equation*}
where $c_0$ is defined in Lemma \ref{lem2}.(ii). Note also that
\begin{equation*}
	\displaystyle \Phi(t,x) = -t^{-2m}\partial_t\left(\eta_M^{2p'}(t)\l(t)\right) \phi(x)
\end{equation*}
in the cone defined in \eqref{suppcond}.

Taking $\Phi$ as the test function in the definition of weak solution \eqref{weaksol}, exploiting the compact support condition \eqref{suppcond} on $u$ and integrating by parts, we obtain
\begin{align*}
	\begin{split}
	& \, \e c_0(-\mu) \int_{\R^n} g\phi dx
	+
	\int_0^T\int_{\R^n} |u_t|^p t^{-2m} \eta_M^{2p'} |\l'|\phi  \, dxdt
	\\&+
	\int_0^T\int_{\R^n} |u_t|^p t^{-2m} |\partial_t \eta_M^{2p'}| \l \phi  \, dxdt
	\\
	=&
	\int_0^T\int_{\R^n} u_t
	t^{-2m} \left(-2mt^{-1}\, \partial_t \eta_M^{2p'} \l -2mt^{-1}\, \eta_M^{2p'} \l'
	+
	\partial_t^2 \eta_M^{2p'} \l + 2 \partial_t \eta_M^{2p'} \l' + \eta_M^{2p'} \l''
	\right) \phi
	 \, dxdt
	\\&-
	\int_0^T\int_{\R^n} \nabla u \cdot \nabla\phi \, \partial_t \left(\eta_M^{2p'} \l \right)  \, dxdt
	\\
	=&
	-\e c_0(\mu) \ints f \phi dx
	-2m\int_0^T\int_{\R^n}u_t
	t^{-2m-1} \partial_t \eta_M^{2p'} \l \phi
	 \, dxdt
	\\&+ \int_0^T\int_{\R^n}
	u_t t^{-2m} \left( \partial_t^2 \eta_M^{2p'} \l + 2 \partial_t \eta_M^{2p'} \l' \right) \phi  \, dxdt
	\\&
	\int_0^T\int_{\R^n}  u_t t^{-2m} \eta_M^{2p'} \left( \l'' -2m t^{-1} \l' - t^{2m}\l \right) \phi  \, dxdt .
	\end{split}
\end{align*}
Neglecting the third term in the left hand-side and recalling that $\l$ solve the ODE \eqref{ODE}, it follows that
\begin{equation}\label{thmtep2}
\begin{aligned}
& \, \e C_1
+ \int_0^T\int_{\R^n} |u_t|^p t^{-2m} \eta_M^{2p'}|\l'|\phi  \, dxdt ,
\\
\le&
-2m\int_0^T\int_{\R^n}u_{t}t^{-2m-1}\partial_t\eta_M^{2p'}\l \phi  \, dxdt
\\
&+
\int_0^T\int_{\R^n}u_{t}t^{-2m}\partial_t^2\eta_M^{2p'} \l \phi  \, dxdt
\\
&+2\int_0^T\int_{\R^n}u_{t}t^{-2m}\partial_t\eta_M^{2p'} \l' \phi  \, dxdt
\\
=:& \, I + II + III,
\end{aligned}
\end{equation}
where
\[
C_1 \equiv C_1(m,f,g) := c_0(\mu) \ints f\phi dx + c_0(-\mu) \ints g\phi dx >0
\]
is a positive constant thanks to \eqref{conddata}.

Now we will estimate the three terms $I, II, III$ by H\"{o}lder's inequality. Firstly let us define the functions
\begin{equation*}
\theta(t) :=
\left\{
\begin{aligned}
&0
&& \text{for $t<\tfrac12$,}
\\
&\eta(t) && \text{for $t\ge\tfrac12$,}
\end{aligned}
\right.
\qquad
\theta_{M}(t):=\theta\left(\frac tM \right),
\end{equation*}
for which it is straightforward to check the following relations hold:
\begin{align}
	\label{esteta'}
	|\partial_t \eta_M^{2p'}|
	&\le \frac{2p'}{M} \n{\eta'}_{L^\infty}
	\theta_M^{2p'/p},
	\\
	\label{esteta''}
	|\partial^2_t \eta_M^{2p'}|
	&\le
	\frac{2p'}{M^2} \left[(2p'-1)\n{\eta'}_{L^\infty}^2 + \n{\eta \eta''}_{L^\infty} \right]
	\theta_M^{2p'/p}.
\end{align}
From now on, $C$ will stand for a generic positive constant, independent of $\e$ and $M$, which can change from line to line.

Exploiting estimates in \eqref{testpro} and \eqref{esteta'}, the asymptotic behaviors (iv)-(v) in Lemma\til\ref{lem2} and the finite speed of propagation property \eqref{suppcond}, for $I$ we obtain
\begin{equation}\label{I}
\begin{aligned}
I=&-2m\int_0^T\int_{\R^n}u_{t}t^{-2m-1}\partial_t\eta_M^{2p'} \l \phi  \, dxdt
\\
\le& \, CM^{-2}
\left(\int_{\frac M2}^M\int_{|x|\le\gamma(t)}t^{-2m}|\l'|^{-\frac{1}{p-1}}|\l|^{\frac{p}{p-1}}\phi  \, dxdt \right)^{\frac{1}{p'}}
\left(\int_0^T\int_{\R^n}|u_t|^pt^{-2m}\theta_M^{2p'} |\l'|\phi  \, dxdt \right)^{\frac1p}
\\
\le& \, CM^{-2}
\left( \int_{\frac M2}^M \int_{0}^{1+\frac{t^{m+1}}{m+1}}
t^{-2m + \frac{mp-3m}{2(p-1)}} (1+r)^{\frac{n-1}{2}} \exp\left(r-\frac{t^{m+1}}{m+1}\right)
\, drdt \right)^{\frac{1}{p'}} \\
&\times
\left(\int_0^T\int_{\R^n}|u_t|^pt^{-2m}\theta_M^{2p'} |\l'|\phi  \, dxdt \right)^{\frac1p}
\\
\le& \, CM^{-2 - \frac{3m}{2} + \frac{m}{2p}  +\left[\frac{(m+1)(n-1)}{2}+1\right]\frac{p-1}{p}}
\left(\int_0^T\int_{\R^n}|u_t|^pt^{-2m}\theta_M^{2p'} |\l'|\phi  \, dxdt \right)^{\frac1p}.
\end{aligned}
\end{equation}
Analogously, for $II$ and $III$ we have
\begin{equation}\label{II}
\begin{aligned}
II=&\int_0^T\int_{\R^n}u_{t}t^{-2m}\partial_t^2\eta_M^{2p'} \l \phi \, dxdt \\
\le& \, CM^{-2 - \frac{3m}{2} + \frac{m}{2p}  +\left[\frac{(m+1)(n-1)}{2}+1\right]\frac{p-1}{p}}\\
&\times\left(\int_0^T\int_{\R^n}|u_t|^pt^{-2m}\theta_M^{2p'}|\l'|\phi \, dxdt \right)^{\frac1p},
\end{aligned}
\end{equation}
and
\begin{equation}\label{III}
\begin{aligned}
III=& \, 2\int_0^T\int_{\R^n}u_{t}t^{-2m}\partial_t\eta_M^{2p'} \l' \phi \, dxdt
\\
\le&\, CM^{-1}
\left(\int_{\frac M2}^M\int_{|x|\le\gamma(t)}t^{-2m}|\l'|\phi \, dxdt \right)^{\frac{1}{p'}}
\\
&\times
\left(\int_0^T\int_{\R^n}|u_t|^pt^{-2m}\theta_M^{2p'} |\l'| \phi \, dxdt \right)^{\frac1p}
\\
\le&\, CM^{-1-\frac{m}{2}+\frac{m}{2p} +\left[\frac{(m+1)(n-1)}{2}+1\right]\frac{p-1}{p}}\\
&\times\left(\int_0^T\int_{\R^n}|u_t|^pt^{-2m}\theta_M^{2p'} |\l'| \phi \, dxdt \right)^{\frac1p}.\\
\end{aligned}
\end{equation}
Since $m\ge -1$ is equivalent to
\[
-1-\frac{m}{2}+\frac{m}{2p} \ge -2-\frac{3m}{2} + \frac{m}{2p},
\]
we conclude, by plugging \eqref{I}, \eqref{II} and \eqref{III} in \eqref{thmtep2}, that
	\begin{equation}\label{thmtep4}
	\begin{aligned}
	&C_1\e+ \int_0^T\int_{\R^n}|u_t|^pt^{-2m}\eta_M^{2p'} |\l'|\phi \, dxdt \\
	\le& \, C M^{-1-\frac{m}{2}+\frac{m}{2p} +\left[\frac{(m+1)(n-1)}{2}+1\right]\frac{p-1}{p}}
	\\
	&\times\left(\int_0^T\int_{\R^n}|u_t|^pt^{-2m}\theta_M^{2p'} |\l'|\phi \, dxdt \right)^{\frac1p}.\\
	\end{aligned}
	\end{equation}

Define now the function
\[
Y[w](M):=\int_1^M\left(\int_0^T\int_{\R^n}w(t, x)\theta_\sigma^{2p'}(t) \, dxdt \right)\sigma^{-1}d\sigma
\]
and let us denote for simplicity
\begin{equation*}
	Y(M) := Y \left[|u_t|^pt^{-2m}|\l'(t)|\phi(x)\right]\!(M).
\end{equation*}
From direct computation we see that
\begin{equation}\label{Y}
\begin{aligned}
Y(M)
=&\int_1^M
\left(\int_0^T\int_{\R^n}|u_t|^pt^{-2m}|\l'(t)|\phi(x)\theta_\sigma^{2p'}(t) \, dxdt \right)\sigma^{-1}d\sigma\\
=&\int_0^T\int_{\R^n}|u_t|^pt^{-2m}|\l'(t)|\phi(x)\int_1^M\theta^{2p'}(t/\sigma)\sigma^{-1}d\sigma  \, dxdt \\
=&\int_0^T\int_{\R^n}|u_t|^pt^{-2m}|\l'(t)|\phi(x)\int_{\frac tM}^t\theta^{2p'}(s)s^{-1}ds  \, dxdt \\
\le&\int_{0}^T\int_{\R^n}|u_t|^pt^{-2m}|\l'(t)|\phi(x)\eta^{2p'}\left(\frac tM\right)\int_{\frac 12}^1 s^{-1}ds  \, dxdt \\
=& \ln2\int_0^T\int_{\R^n} |u_t|^p t^{-2m} \eta_M^{2p'}(t) |\l'(t)|\phi(x) \, dxdt,
\end{aligned}
\end{equation}
where we used the definition of $\theta(t)$. Moreover
\begin{equation}\label{Yderiv}
\begin{aligned}
Y'(M)=\frac{d}{dM}Y(M)=M^{-1}\int_0^T\int_{\R^n}|u_t|^p t^{-2m} \theta_M^{2p'}(t) |\l'(t)|\phi(x) \, dxdt .\\
\end{aligned}
\end{equation}
Hence by combining \eqref{thmtep4}, \eqref{Y} and \eqref{Yderiv}, we get
\begin{equation*}\label{IneqforY}
\begin{aligned}
M^{\left[\frac{(m+1)(n-1)-m}{2}\right](p-1)} Y'(M) \ge \left[C_1\e+ (\ln2)^{-1} Y(M)\right]^p,
\end{aligned}
\end{equation*}
which leads to
\begin{equation*}
\label{lifespanM}
M\le
 \left\{
 \begin{array}{ll}
 C\e^{-\left(\frac{1}{p-1}+\frac{m-(m+1)(n-1)}{2}\right)^{-1}} & \mbox{for}\ 1<p<p_T(n,m),\\
 \exp\left(C\e^{-(p-1)}\right) & \mbox{for}\ p=p_T(n,m).
 \end{array}
 \right.
\end{equation*}
Since $M$ is arbitrary in $(1, T)$, we finally obtain the lifespan estimates \eqref{lifespan2}.
\end{proof}

\section*{Appendices}

\appendix

\section{Some formulas for the modified Bessel functions}\label{appA}
For the reader's convenience, here we collect some formulas from Section~9.6 and Section~9.7 of the handbook by Abramowitz and Stegun \cite{AS}.
\begin{itemize}
	\item
	The solutions to the differential equation
	\begin{equation*}
	z^2 \frac{d^2}{dz^2} w(z) + z \frac{d}{dz} w(z) - \left(z^2 + \nu^2 \right) w(z) = 0
	\end{equation*}
	are the modified Bessel functions $I_{\pm\nu}(z)$ and $K_{\nu}(z)$.	
	$I_\nu(z)$ and $K_\nu(z)$ are real and positive when $\nu>-1$ and $z>0$.
	
	\item Relations between solutions:
	\begin{equation}
	\label{ch2A2}
	K_\nu(z) = K_{-\nu}(z) = \frac{\pi}{2} \frac{I_{-\nu}(z)-I_{\nu}(z)}{\sin(\nu\pi)} .
	\end{equation}
	When $\nu\in\Z$, the right hand-side of this equation is replaced by its limiting value. 
	
	\item Ascending series:
	\begin{equation}
	\label{ch2A4}
	I_{\nu}(z) = \sum_{k=0}^{\infty} \frac{(z/2)^{2k+\nu}}{k! \, \Gamma(k+1+\nu)},
	\end{equation}
	where $\Gamma$
	is the Gamma function.
	
	\item Wronskian:
	\begin{gather}
	\label{ch2wr2}
	I_{\nu}(z) K_{\nu+1}(z) + I_{\nu+1}(z) K_{\nu}(z) = \frac{1}{z}.
	\end{gather}
	
	\item Recurrence relations:
	\begin{align}
	\label{ch2A6}
	\partial_z I_{\nu}(z) &= I_{\nu+1}(z) + \frac{\nu}{z} I_{\nu}(z), &
	\partial_z K_{\nu}(z) &= -K_{\nu+1}(z) + \frac{\nu}{z} K_{\nu}(z), \\
	\label{ch2A7}
	\partial_z I_{\nu}(z) &= I_{\nu-1}(z) - \frac{\nu}{z} I_{\nu}(z), &
	\partial_z K_{\nu}(z) &= -K_{\nu-1}(z) - \frac{\nu}{z} K_{\nu}(z).
	\end{align}

	\item Limiting forms for fixed $\nu$ and $z \to 0$:
	\begin{alignat}{2}
	\label{limitI0}
	I_{\nu}(z) &\sim \frac{1}{\Gamma(\nu+1)} \left(\frac{z}{2}\right)^{\nu}
	&\quad &\text{for $\nu\neq -1,-2,\dots$}
	\\
	\label{limit00}
	K_0(z) &\sim -\ln(z), &\quad &\text{}
	\\
	\label{limit0}
	K_\nu(z) &\sim \frac{\Gamma(\nu)}{2} \left(\frac{z}{2}\right)^{-\nu},
	&\quad &\text{for $\Re\nu>0$.}
	\end{alignat}

	\item Asymptotic expansions for fixed $\nu$ and large arguments:
	\begin{align}
	\label{ch2A11}
	&I_\nu (z)
	= \frac{1}{\sqrt{2\pi }} z^{-1/2} e^z \times (1+O(z^{-1})),
	&&{\text{for $|z|$ large and $|\arg z| < \frac{\pi}{2}$,}}
	\\
	\label{ch2A12}
	&K_\nu (z)
	= \sqrt{\frac{\pi}{2}} z^{-1/2} e^{-z}  \times (1+O(z^{-1})),
	&&\text{for $|z|$ large and $|\arg z| < \frac{3}{2}\pi$.}
	\end{align}

\end{itemize}

\section{Solution formula for the ODE}\label{appB}
	We show in this section how to discover the formula of the solution for equation \eqref{ODE}. Let us suppose $m\in\N$ and make the ansatz
	\begin{equation}\label{ansatzl}
		\l(t) = \sum_{h=0}^{\infty} a_h t^h,
	\end{equation}
	for some constants $\{a_h\}_{h\in\N}$. Hence,
	\begin{equation*}
		\l'(t) = \sum_{h=1}^{\infty} h a_h t^{h-1},
		\qquad
		\l''(t) = \sum_{h=2}^{\infty} h(h-1) a_h t^{h-2}.
	\end{equation*}
	Substituting in \eqref{ODE} and multiplying by $t^2$, we get
	\begin{equation}\label{eq1}
	\begin{split}
		0 &=
		\sum_{h=2}^\infty h(h-1) a_h t^{h}
		-2m \sum_{h=1}^{\infty} h a_h t^{h}
		- \sum_{h=0}^{\infty} a_h t^{h+2m+2}
		\\
		&= \sum_{h=2}^\infty h(h-1) a_h t^{h}
		-2m \sum_{h=1}^{\infty} h a_h t^{h}
		- \sum_{h=2m+2}^{\infty} a_{h-2m-2} t^{h}
		\\
		&= \sum_{h=1}^{2m} h(h-2m-1) a_h t^{h} +
		\sum_{2m+2}^{\infty} [h(h-2m-1)a_h - a_{h-2m-2}] t^{h}.
	\end{split}
	\end{equation}
	Let us fix the constant $a_0$ and $a_{2m+1}$. We will write the other constants in dependence of these ones.
	Indeed, we infer from \eqref{eq1} that
	\begin{equation*}
		\begin{aligned}
		&a_h =0 &&\quad\text{for $h=1,\dots,2m$,}
		\\
		&a_h = \frac{a_{h-2m-2}}{h(h-2m-1)}
		&&\quad\text{for $h \ge 2m+2$.}
		\end{aligned}
	\end{equation*}
	Hence, by an inductive argument, we can prove that
	\begin{equation*}
	\begin{split}
		a_h &=
		\left\{
		\begin{aligned}
		& \frac{a_0}{[2(m+1)]^k \, k! \, \prod_{j=1}^{k} [2(m+1)j-(2m+1)]}
		&&\quad\text{if $h=2(m+1)k$, $k\in\N$,}
		\\
		& \frac{a_{2m+1}}{[2(m+1)]^k \, k! \, \prod_{j=1}^{k} [2(m+1)j+(2m+1)]}
		&&\quad\text{if $h=2(m+1)k+2m+1$, $k\in\N$,}
		\\
		&0 &&\quad\text{otherwise,}
		\end{aligned}
		\right.
		\\
		&=
		\left\{
		\begin{aligned}
		& \frac{[2(m+1)]^{-2k} \, \Gamma\left(1-\frac{m+1/2}{m+1}\right) }{k! \, \Gamma\left(k+1-\frac{m+1/2}{m+1}\right)} a_0
		&&\quad\text{if $h=2(m+1)k$, $k\in\N$,}
		\\
		& \frac{[2(m+1)]^{-2k} \, \Gamma\left(1+\frac{m+1/2}{m+1}\right) }{k! \, \Gamma\left(k+1+\frac{m+1/2}{m+1}\right)}a_{2m+1}
		&&\quad\text{if $h=2(m+1)k+2m+1$, $k\in\N$,}
		\\
		&0 &&\quad\text{otherwise,}
		\end{aligned}
		\right.
	\end{split}
	\end{equation*}
	where we used the relations
	\begin{equation*}
		\prod_{j=1}^{k} (cj \pm 1) = c^k \frac{\Gamma(k+1 \pm 1/c)}{\Gamma(1 \pm 1/c)}.
	\end{equation*}
	Substituting the values of $a_h$ into \eqref{ansatzl}, we have
	\begin{equation*}
	\begin{split}
		\l(t)
		=&\
		a_0 \Gamma\left( 1-\frac{m+1/2}{m+1} \right)
		\sum_{k=0}^{\infty}
		\frac{[2(m+1)]^{-2k}}{k! \, \Gamma\left(k+1-\frac{m+1/2}{m+1}\right)} t^{2(m+1)k}
		 \\ &+
		a_{2m+1} \Gamma\left( 1+\frac{m+1/2}{m+1} \right)
		t^{2m+1}
		\sum_{k=0}^{\infty}
		\frac{[2(m+1)]^{-2k}}{k! \, \Gamma\left(k+1+\frac{m+1/2}{m+1}\right)} t^{2(m+1)k}
		 \\ =&\
		c_- a_0
		t^{m+1/2}
		\sum_{k=0}^{\infty}
		\frac{1}{k! \, \Gamma\left(k+1-\frac{m+1/2}{m+1}\right)} \left[\frac{t^{m+1}}{2(m+1)}\right]^{2k-\frac{m+1/2}{m+1}}
		\\ &+
		c_+ a_{2m+1}
		t^{m+1/2}
		\sum_{k=0}^{\infty}
		\frac{1}{k! \, \Gamma\left(k+1+\frac{m+1/2}{m+1}\right)} \left[\frac{t^{m+1}}{2(m+1)}\right]^{2k+\frac{m+1/2}{m+1}},
	\end{split}
	\end{equation*}
	with
	\begin{equation*}
		c_\pm = \Gamma\left( 1 \pm \frac{m+1/2}{m+1} \right)
		[2(m+1)]^{\pm \frac{m+1/2}{m+1}}.
	\end{equation*}
	Taking into account the relations \eqref{ch2A4} and \eqref{ch2A2} we get
	\begin{equation*}
	\begin{split}
		\l(t) &=
		c_- a_0 t^{m+1/2} I_{-\frac{m+1/2}{m+1}} \left(\frac{t^{m+1}}{m+1}\right)
		+
		c_+ a_{2m+1} t^{m+1/2} I_{\frac{m+1/2}{m+1}} \left(\frac{t^{m+1}}{m+1}\right)
		\\
		&= k_1 t^{m+1/2} I_{\frac{m+1/2}{m+1}} \left(\frac{t^{m+1}}{m+1}\right)
		+
		k_2
		t^{m+1/2}
		K_{\frac{m+1/2}{m+1}} \left(\frac{t^{m+1}}{m+1}\right)
	\end{split}
	\end{equation*}
	with
	\begin{equation*}
		k_1 = c_- a_0 + c_+ a_{2m+1},
		\qquad
		k_2 = \frac{2}{\pi} c_- a_0 \sin\left(\frac{m+1/2}{m+1} \pi \right).
	\end{equation*}
	In this way we can deduce the fundamental solutions of the equation \eqref{ODE} when $m\in\N$, and from Lemma\til\ref{lem1} we know they also hold for $m \in \R\setminus\{-1\}$.

\section{Proof of Theorem\til\ref{thm:estW}}
\label{app:YRest}

In this appendix we prove the $L^p-L^q$ estimates on the conjugate line for $W_1(s,t,D_x)$, $W_2(s,t,D_x)$ and their derivatives respect to time collected in Theorem\til\ref{thm:estW}.
The argument is adapted from the proof of Theorem 3.3 by Yagdjian \cite{Yag2} (see also \cite{Rei97a} and \cite[Chapter 16]{ER18}), where similar estimates for $V_1(t,D_x)$ and $V_2(t,D_x)$ are presented. Note that in \cite{Yag2} the additional hypothesis $\sigma\ge0$ is supposed, but this can be dropped, as we will show.

Before to proceed, we recall the following key lemmata.

\begin{Def}
	Denote by $L^q_p \equiv L^q_p(\R^n)$ the space of tempered distributions $T$ such that
	\begin{equation*}
	\n{T * f}_{L^q} \le C \n{f}_{L^p}
	\end{equation*}
	for all Schwartz functions $f\in\mathcal{S}(\R^n)$ and a suitable positive constant $C$ independent of $f$.
	
	Denote instead with $M^q_p \equiv M^q_p(\R^n)$ the set of multiplier of type $(p,q)$, i.e. the set of Fourier transforms $\F(T)$ of distributions $T \in L^q_p$.
\end{Def}

\begin{lem}[\cite{Hor60}, Theorem 1.11 ]
	\label{lem:hormander}
	Let $f$ be a measurable function such that for all positive $\l$, we have
	\begin{equation*}
	\meas\{\xi \in \R^n \colon |f(\xi)| \le \l \} \le C \l^{-b}
	\end{equation*}
	for some suitable $b \in (1,\infty)$ and positive $C$. Then, $f \in M^q_p$ if $1<p\le2\le q < \infty$ and $1/p-1/q=1/b$.
\end{lem}

\begin{lem}[\cite{Bre75}, Lemma 2]
	\label{lem:brenner}
	Fix a nonnegative smooth function $\chi \in C^\infty_0([0,\infty))$ with compact support $\supp\chi \subset \{x \in \R^n \colon 1/2 \le |x| \le 2 \}$ such that $\sum_{k=-\infty}^{\infty} \chi(2^{-k} x)=1$ for $x\neq0$. Set $\chi_k(x) := \chi(2^{-k}x)$ for $k\ge1$ and $\chi_0(x) := 1 - \sum_{k=1}^{\infty} \chi_k(x)$, so that $\supp\chi_0 \subset \{x \in \R^n \colon |x| \le 2 \}$.
	
	Let $a\in L^\infty(\R^n)$, $1<p\le2$ and assume that
	\begin{equation*}
	\n{ \F^{-1}(a \chi_k \widehat{v})}_{L^{p'}}
	\le C \n{v}_{L^p}
	\quad
	\text{for $k\ge0$.}
	\end{equation*}
	Then for some constant $A$ independent of $a$ we have
	%
	\begin{equation*}
	\n{ \F^{-1}(a \widehat{v})}_{L^{p'}}
	\le AC \n{v}_{L^{p}}.
	\end{equation*}	
\end{lem}

\begin{lem}[Littman type Lemma, see Lemma 4 in \cite{Bre75}]
	\label{lem:littman}
	Let $P$ be a real function, smooth in a neighbourhood of the support of $v \in C_0^\infty(\R^n)$. Assume that the rank of the Hessian matrix $(\partial^2_{\eta_j \eta_k} P(\eta))_{j,k\in\{1,\dots,n\}}$ is at least $\rho$ on the support of $v$. Then for some integer $N$ the following estimate holds:
	\begin{equation*}
	\n{\F^{-1}(e^{itP(\eta)}v(\eta))}_{L^\infty} \le C (1+|t|)^{-\rho/2} \sum_{|\alpha| \le N} \n{\partial_\eta^\alpha v}_{L^1}.
	\end{equation*}
\end{lem}

We will prove now only estimate (iii) of Theorem\til\ref{thm:estW}, since the computation for estimates (i) and (ii) are completely analogous; about estimate (iv), we will sketch the proof since it could be strange to the reader that this is the only case where the range of $\sigma$ collapses to be only a value.

	First of all, let us set $\tau := t/s \ge 1$, $z=2i\phi(t)\xi$, $\zeta=2i\phi(s)\xi$ and let us introduce the smooth functions $X_0, X_1, X_2 \in C^\infty(\R^n;[0,1])$ satisfying
	\begin{align*}
	X_0(x) &=
	\begin{cases}
	1 & \text{for $|x| \le 1/2$,} \\
	0 & \text{for $|x| \ge 3/4 $,}
	\end{cases}
	\\
	X_2(x) &=
	\begin{cases}
	1 & \text{for $ |x| \ge 1$,} \\
	0 & \text{for $ |x| \le 3/4$,}
	\end{cases}
	\\
	X_1(x) &= 1- X_0(\tau^{m+1} x) - X_2(x) .
	\end{align*}
	In particular, observe that \[X_0(\phi(t)\xi) + X_1(\phi(s)\xi) + X_2(\phi(s)\xi) \equiv 1\] for $0<s\le t$ and  $\xi\in\R^n$.

	By relations \eqref{derPhi1} and \eqref{derPhi2}, it is straightforward to get
	\begin{equation*}
	\begin{split}
	\partial_t V_1(t,|\xi|) =&\,
	\frac{m+1}{2} t^{-1} z e^{-z/2} [\Phi(\mu+1,2\mu+1;z)-\Phi(\mu,2\mu;z)]
	\\
	\partial_t V_2(t,|\xi|) =&\,
	e^{-z/2} \left[ \Phi(1-\mu,1-2\mu;z) - \frac{m+1}{2} z \Phi(1-\mu,2(1-\mu);z)\right].
	\end{split}
	\end{equation*}
	Thus one can check, using identity \eqref{PhiH}, that
	\begin{align*}
	\partial_t W_1(s,t,|\xi|) =&\,
	i s t^{m} e^{-(z+\zeta)/2} |\xi| [\Phi(\mu+1,2\mu+1;z)-\Phi(\mu,2\mu;z)]
	\Phi(1-\mu,2(1-\mu);\zeta)
	\\
	=&\,		
	i s t^{m} e^{-\zeta/2} |\xi| [e^{z/2} {H}^0_+(z) + e^{-z/2} {H}^0_-(z) +]
	\Phi(1-\mu,2(1-\mu);\zeta)
	\\
	=&\,		
	i s t^{m} |\xi| \left[e^{[1+\tau^{m+1}]\zeta/2} {H}^0_+(z) {H}^1_+(\zeta)
	+
	e^{-[1-\tau^{m+1}]\zeta/2} {H}^0_+(z) {H}^1_-(\zeta)
	\right.
	\\
	&\qquad\qquad
	\left.
	+
	e^{[1-\tau^{m+1}]\zeta/2} {H}^0_-(z) {H}^1_+(\zeta)
	+
	e^{-[1+\tau^{m+1}]\zeta/2} {H}^0_-(z) {H}^1_-(\zeta)
	\right]
	%
	\\
	%
	\partial_t W_2(s,t,|\xi|) =&\,
	e^{-(z+\zeta)/2} \Phi(\mu,2\mu;\zeta)\left[ \Phi(1-\mu,1-2\mu;z) - i t^{m+1} |\xi| \Phi(1-\mu,2(1-\mu);z)\right]
	\\
	=&\,
	e^{-\zeta/2} \Phi(\mu,2\mu;\zeta)
	[e^{z/2} H^2_+(z) + e^{-z/2} H^2_-(z)]
	\\
	=&\,
	e^{[1+\tau^{m+1}]\zeta/2} {H}^2_+(z) {H}^3_+(\zeta)
	+
	e^{-[1-\tau^{m+1}]\zeta/2} {H}^2_+(z) {H}^3_-(\zeta)
	\\
	&\quad
	+
	e^{[1-\tau^{m+1}]\zeta/2} {H}^2_-(z) {H}^3_+(\zeta)
	+
	e^{-[1+\tau^{m+1}]\zeta/2} {H}^2_-(z) {H}^3_-(\zeta)
	\end{align*}
	where for the simplicity we set
	\begin{align*}
	{H}^0_\pm(z) :=&\,
	\frac{\Gamma(2\mu+1)}{\Gamma\left(\mu+\frac{1}{2} \pm \frac{1}{2} \right)} H_\pm(\mu+1,2\mu+1;z) - \frac{\Gamma(2\mu)}{\Gamma(\mu)} H_\pm(\mu,2\mu;z)
	,
	\\
	{H}^1_{\pm}(\zeta) :=&\,
	\frac{\Gamma(2(1-\mu))}{\Gamma(1-\mu)} {H}_{\pm}(1-\mu,2(1-\mu);\zeta),
	\end{align*}
	and
	\begin{align*}
	H^2_\pm(z) :=&\,
	\frac{\Gamma(1-2\mu)}{\Gamma\left(\frac{1}{2}\pm\frac{1}{2}-\mu\right)} H_\pm(1-\mu,1-2\mu;z)
	\\
	&
	-i t^{m+1} |\xi|
	\frac{\Gamma(2(1-\mu))}{\Gamma(1-\mu)}
	H_\pm(1-\mu,2(1-\mu);z) ,
	\\
	H^3_\pm(\zeta) :=&\,
	\frac{\Gamma(2\mu)}{\Gamma(\mu)} H_\pm(\mu,2\mu;\zeta).
	\end{align*}
	
	
	\textbf{Estimates at low frequencies for $\partial_t W_1(s,t,D_x)$.} Let us consider the Fourier multiplier
	\begin{equation*}
	\F^{-1}_{\xi\to x}\left(X_0(\phi(t)\xi) |\xi|^{-\sigma} \partial_t W_1(s,t,|\xi|) \widehat{\psi}\right).
	\end{equation*}
	By the change of variables $\eta := \phi(t)\xi$ and $x:=\phi(t)y$ we get
	\begin{multline*}
	\n{F^{-1}_{\xi\to x} \left( X_0(\phi(t)\xi) |\xi|^{-\sigma} \partial_t W_1(s,t,|\xi|) \widehat{\psi} \right)}_{L^{q'}}
	\\
	\lesssim
	\tau^{-1} t^{(n/q'-n+\sigma)(m+1)}
	\n{T_0 * \F^{-1}_{\eta\to y}\left(\widehat{\psi}(\eta/\phi(t))\right)}_{L^{q'}}
	\end{multline*}
	where
	\begin{align*}
	T_0 :=&\, \F^{-1}_{\eta\to y} \left(
	X_0(\eta) |\eta|^{1-\sigma} e^{-i[1+1/\tau^{m+1}]|\eta|}
	{\Phi}_0(\mu;\tau;|\eta|)
	\right)
	\\
	{\Phi}_0(\mu;\tau;|\eta|)
	:=&\,
	[\Phi(\mu+1,2\mu+1;2i|\eta|)-\Phi(\mu,2\mu;2i|\eta|)]
	\\
	&\times \Phi(1-\mu,2(1-\mu);2i|\eta|/\tau^{m+1})
	\\
	=&\,
	O(|\eta|)[1+\tau^{-(m+1)}O(|\eta|)].
	\end{align*}
	The last equality above is implied by \eqref{Phi0}, from which we deduce $|{\Phi}_0(\mu;\tau;|\eta|)|\lesssim 1$ if $|\eta|\le 3/4$. So, for any $\l>0$, we obtain
	\begin{align*}
	\meas\{ \eta\in\R^n \colon |\F_{y \to \eta}(T_0)| \ge \l\}
	&\le
	\meas\{ \eta\in\R^n \colon \text{$|\eta| \le 3/4$ and $|\eta|^{1-\sigma} \gtrsim \l$} \}
	\\
	&\lesssim
	\begin{cases}
	1 & \text{if $0 < \l \le 1$,} \\
	0 & \text{if $\l\ge 1$ and $\sigma \le 1$,} \\
	\l^{-\frac{n}{\sigma-1}} & \text{if $\l\ge 1$ and $\sigma > 1$,}
	\end{cases}
	\\
	&\lesssim \l^{-b},
	\end{align*}
	where $1<b<\infty$ if $\sigma \le 1$ and $1 <b \le \frac{n}{\sigma-1}$ if $\sigma>1$.
	Hence by Lemma \ref{lem:hormander}, we get $T_0 \in L^{q'}_q$ for $1<q\le 2 \le q'<\infty$ and $\sigma \le 1 + n (\frac{1}{q}-\frac{1}{q'})$.
	Then we obtain the Hardy-Littlewood type inequality
	\begin{equation*}
	\n{F^{-1}_{\xi\to x} \left( X_0(\phi(t)\xi) |\xi|^{-\sigma} \partial_t W_1(s,t,|\xi|) \widehat{\psi} \right)}_{L^{q'}}
	\lesssim
	\tau^{-1} t^{\left[\sigma-n \left( \frac{1}{q} - \frac{1}{q'} \right) \right](m+1)}
	\n{\psi}_{L^q}.
	\end{equation*}
	Observing that, by the assumption on the range of $\sigma$,
	\begin{align*}
	\tau^{-1} t^{\left[\sigma-n \left( \frac{1}{q} - \frac{1}{q'} \right) \right](m+1)}
	&=
	\tau^{-\left[1-\mu + n \left( \frac{1}{q} - \frac{1}{q'} \right)  -\sigma \right](m+1)}
	\tau^{m/2} s^{\left[\sigma-n \left( \frac{1}{q} - \frac{1}{q'} \right) \right](m+1)}
	\\&\le
	\tau^{m/2} s^{\left[\sigma-n \left( \frac{1}{q} - \frac{1}{q'} \right) \right](m+1)},
	\end{align*}
	we finally get
	\begin{equation}\label{Zpd}
	\n{F^{-1}_{\xi\to x} \left( X_0(\phi(t)\xi) |\xi|^{-\sigma} \partial_t W_1(s,t,|\xi|) \widehat{\psi} \right)}_{L^{q'}}
	\lesssim
	\tau^{m/2} s^{\left[\sigma-n \left( \frac{1}{q} - \frac{1}{q'} \right) \right](m+1)}
	\n{\psi}_{L^q}.
	\end{equation}
	
	
	\textbf{Estimates at intermediate frequencies for $\partial_t W_1(s,t,D_x)$.}
	We proceed similarly as before. Let us consider now the Fourier multiplier
	\begin{equation*}
	\F^{-1}_{\xi\to x}\left(X_1(\phi(s)\xi) |\xi|^{-\sigma} \partial_t W_1(s,t,|\xi|) \widehat{\psi}\right).
	\end{equation*}
	Exploiting this time the change of variables $\eta := \phi(s)\xi$ and $x:=\phi(s)y$, we get
	\begin{multline*}
	\n{F^{-1}_{\xi\to x} \left( X_1(\phi(s)\xi) |\xi|^{-\sigma} \partial_t W_1(s,t,|\xi|) \widehat{\psi} \right)}_{L^{q'}}
	\\
	\lesssim
	\tau^{m/2} s^{(n/q'-n+\sigma)(m+1)}
	\n{T_1 * \F^{-1}_{\eta\to y}\left(\widehat{\psi}(\eta/\phi(s))\right)}_{L^{q'}}
	\end{multline*}
	where
	\begin{align*}
	T_1 :=&\, \F^{-1}_{\eta\to y} \left(
	X_1(\eta) |\eta|^{1-\sigma} e^{-i[1+\tau^{m+1}]|\eta|}
	{\Phi}_1(\mu;\tau;|\eta|)
	\right)
	\\
	{\Phi}_1(\mu;\tau;|\eta|)
	:=&\,
	\tau^{m/2}
	[\Phi(\mu+1,2\mu+1;2i\tau^{m+1}|\eta|)-\Phi(\mu,2\mu;2i\tau^{m+1}|\eta|)]
	\\
	&\times \Phi(1-\mu,2(1-\mu);2i|\eta|).
	\end{align*}
	Taking in account \eqref{Phi0} and \eqref{estPhi}, we infer that
	\begin{equation*}
	|\Phi_1(\mu;\tau;|\eta|)| \lesssim
	\tau^{m/2} (\tau^{m+1}|\eta|)^{-\mu} =
	|\eta|^{-\mu}
	\quad
	\text{ on $\supp X_1(\eta) \subseteq \left[(2\tau^{m+1})^{-1},1 \right]$, }
	\end{equation*}
	and thus, for any $\l>0$, we obtain
	\begin{align*}
	\meas\{ \eta\in\R^n \colon |\F_{y \to \eta}(T_1)| \ge \l\}
	&\le
	\meas\{ \eta\in\R^n \colon \text{$|\eta| \le 1$ and $|\eta|^{1-\mu-\sigma} \gtrsim \l$} \}
	\\
	&\lesssim
	\begin{cases}
	1 & \text{if $0 < \l \le 1$,} \\
	0 & \text{if $\l\ge 1$ and $\sigma \le 1-\mu$,} \\
	\l^{-\frac{n}{\sigma-1+\mu}} & \text{if $\l\ge 1$ and $\sigma > 1-\mu$,}
	\end{cases}
	\\
	&\lesssim \l^{-b},
	\end{align*}
	where $1<b<\infty$ if $\sigma \le 1-\mu$ and $1<b \le \frac{n}{\sigma-1+\mu}$ if $\sigma>1-\mu$.
	Hence by Lemma \ref{lem:hormander}, we get $T \in L^{q'}_q$ for $1<q\le 2 \le q'<\infty$ and $\sigma \le 1-\mu + n (\frac{1}{q}-\frac{1}{q'})$.
	Then we reach
	\begin{equation}\label{Zinter}
	\n{F^{-1}_{\xi\to x} \left( X_1(\phi(s)\xi) |\xi|^{-\sigma} \partial_t W_1(s,t,|\xi|) \widehat{\psi} \right)}_{L^{q'}}
	\lesssim
	\tau^{m/2} s^{\left[\sigma-n \left( \frac{1}{q} - \frac{1}{q'} \right) \right](m+1)}
	\n{\psi}_{L^q}.
	\end{equation}

	\textbf{Estimates at high frequencies for $\partial_t W_1(s,t,D_x)$.}
	Finally, we want to estimate the Fourier multiplier
	\begin{equation*}
	\F^{-1}_{\xi\to x}\left(X_2(\phi(s)\xi) |\xi|^{-\sigma} \partial_t W_1(s,t,|\xi|) \widehat{\psi}\right).
	\end{equation*}
	We choose a set of functions $\{ \chi_k \}_{k\ge0}$ as in the statement of Lemma \ref{lem:brenner}.
	
	\textit{$L^1-L^\infty$ estimates.}
	We claim that, for $k\ge0$,
	\begin{equation}\label{est1oo}
	\n{\F^{-1}_{\xi\to x}\left( X_2(\phi(s)\xi) \, \chi_k(\phi(s) \xi) \, |\xi|^{-\sigma} \partial_t W_1(s,t,|\xi|) \right)}_{L^\infty}
	\lesssim
	2^{k\left( n-\sigma \right)} \tau^{m/2} s^{(\sigma-n)(m+1)}.
	\end{equation}
	Exploiting the change of variables $\phi(s)\xi=2^k \eta$ and $2^k x=\phi(s)y$, by the expression of the symbol $\partial_t W_1(s,t,|\xi|)$, we obtain
	\begin{multline}\label{1oo}
	\n{\F^{-1}_{\xi\to x}\left( X_2(\phi(s)\xi) \, \chi_k(\phi(s) \xi) \, |\xi|^{-\sigma} \partial_t W_1(s,t,|\xi|) \right)}_{L^\infty}
	\\\lesssim
	2^{k(n-\sigma+1)} \tau^m s^{(\sigma-n)(m+1)}
	[A^+_+ + A^+_- + A^-_+ + A^-_-]
	\end{multline}
	where
	\begin{align*}
	A^+_\pm :=&\, \n{\F^{-1}_{\eta\to y} \left( e^{i[ \pm 1 + \tau^{m+1} ]2^k|\eta|} v^{+,\pm}_k(\eta) \right) }_{L^\infty},
	\\
	A^-_\pm :=&\, \n{\F^{-1}_{\eta\to y} \left( e^{i[ \pm 1 - \tau^{m+1}]2^k|\eta|} v^{-,\pm}_k(\eta) \right) }_{L^\infty},
	\end{align*}
	and
	\begin{align*}
	v^{+,\pm}_k(\eta) :=&\,
	X_2(2^k \eta)\chi(\eta) |\eta|^{1-\sigma}
	{H}^0_+(2i\tau^{m+1}2^{k}|\eta|) {H}^1_\pm(2i2^{k}|\eta|),
	\\
	v^{-,\pm}_k(\eta) :=&\,
	X_2(2^k \eta)\chi(\eta) |\eta|^{1-\sigma}
	{H}^0_-(2i\tau^{m+1}2^{k}|\eta|) {H}^1_\pm(2i2^{k}|\eta|).
	\end{align*}
	The functions $v^{\pm,\pm}_k(\eta)$ are smooth and compactly supported on $\{\eta \in \R^n \colon 1/2 \le |\eta| \le 2\}$.
	When $k=0$, it is easy to see by estimates \eqref{estH+}-\eqref{estH-} that
	\begin{equation*}
	\n{\F^{-1}_{\eta\to y} \left( e^{i[\pm 1+\tau^{m+1}]|\eta|} v^{+,\pm}_0(\eta) \right) }_{L^\infty}
	\le
	\n{v^{+,\pm}_0}_{L^1}
	\lesssim \tau^{-(m+1)\mu} \n{X_2(\eta)\chi(\eta)|\eta|^{-\sigma}}_{L^1}
	\lesssim \tau^{- m/2}.
	\end{equation*}
	For $k\ge1$, by Lemma \ref{lem:littman} we have, for some integer $N>0$, that
	\begin{equation}\label{1-oomult}
	\n{\F^{-1}_{\eta\to y} \left( e^{i[\pm 1+\tau^{m+1}]2^k|\eta|} v^{+,\pm}_k(\eta) \right) }_{L^\infty}
	\lesssim
	(1+[\pm1 + \tau^{m+1}]2^k)^{-\frac{n-1}{2}} \sum_{|\alpha|\le N} \n{\partial_\eta^\alpha v^{+,\pm}_k}_{L^1}
	.
	\end{equation}
	Since $X_2(2^k\eta)\chi(\eta)=\chi(\eta)$ for $k\ge1$, by estimates \eqref{estH+}-\eqref{estH-} and Leibniz rule we infer
	\begin{align*}
	|\partial^\alpha_\eta v^{+,\pm}_k(\eta)|
	&=
	\left\lvert
	\sum_{\gamma\le\beta\le\alpha}
	\binom{\alpha}{\beta}\binom{\beta}{\gamma}
	\partial^{\alpha-\beta}_\eta
	\left(\chi(\eta)|\eta|^{1-\sigma} \right)
	\partial^{\beta-\gamma}_\eta {H}^0_+(2i\tau^{m+1}2^{k}|\eta|)
	\partial_\eta^\gamma
	{H}^1_\pm(2i2^{k}|\eta|)
	\right\rvert
	\\
	&\lesssim
	\tau^{-m/2}
	2^{-k}
	\sum_{\beta\le\alpha}
	C_{\mu,\alpha,\beta}
	\mathbf{1}_{\left[{1}/{2},2\right]}(\eta)
	|\eta|^{-1-|\beta|}
	\end{align*}
	where $\mathbf{1}_{\left[{1}/{2},2\right]}(\eta)=1$ for $1/2\le |\eta| \le 2$ and $\mathbf{1}_{\left[{1}/{2},2\right]}(\eta)=0$ otherwise.
	From the latter estimate and \eqref{1-oomult}, we get
	\begin{equation*}
	A^+_\pm
	\lesssim
	\tau^{-m/2} 2^{-k} (1+[\pm1+\tau^{m+1}])^{-\frac{n-1}{2}}
	\le
	\tau^{-m/2} 2^{-k}.
	\end{equation*}
	Similarly we obtain also that $A^-_\pm \lesssim \tau^{-m/2} 2^{-k}$. Thus, inserting in  \eqref{1oo} we obtain \eqref{est1oo}, which combined with the Young inequality give us the $L^1-L^\infty$ estimate
	\begin{multline}\label{Zh1oo}
	\n{\F^{-1}_{\xi\to x}\left( X_2(\phi(s)\xi) \, \chi_k(\phi(s) \xi) \, |\xi|^{-\sigma} \partial_t W_1(s,t,|\xi|) \widehat{\psi}\right)}_{L^\infty}
	\\
	\lesssim
	2^{k\left( n-\sigma \right)} \tau^{m/2} s^{(\sigma-n)(m+1)} \n{\psi}_{L^1}.
	\end{multline}
	
	\textit{$L^2 - L^2$ estimates.}
	By the Plancherel formula, H\"older inequality, estimate \eqref{estPhi} and the substitution $\phi(s)\xi=2^k\eta$, we obtain
	\begin{multline}\label{Zh22}
	\n{\F^{-1}_{\xi\to x}\left( X_2(\phi(s)\xi) \, \chi_k(\phi(s) \xi) \, |\xi|^{-\sigma} \partial_t W_1(s,t,|\xi|) \widehat{\psi}\right)}_{L^2}
	\\
	\begin{split}
	&\le
	\n{ X_2(\phi(s)\xi) \, \chi_k(\phi(s) \xi) \, |\xi|^{-\sigma} \partial_t W_1(s,t,|\xi|) }_{L^\infty} \n{\psi}_{L^2}
	\\
	&\lesssim
	2^{-k\sigma} \tau^{m/2} s^{\sigma(m+1)} \n{\psi}_{L^2}.
	\end{split}
	\end{multline}
	
	\textit{$L^q-L^{q'}$ estimates.} The interpolation between \eqref{Zh1oo} and \eqref{Zh22} give us the estimates on the conjugate line
	\begin{multline}\label{Zh}
	\n{\F^{-1}_{\xi\to x}\left( X_2(\phi(s)\xi) \, \chi_k(\phi(s) \xi) \, |\xi|^{-\sigma} \partial_t W_1(s,t,|\xi|) \widehat{\psi}\right)}_{L^{q'}}
	\\
	\lesssim
	2^{k\left[ n \left(\frac{1}{q}-\frac{1}{q'} \right) -\sigma \right]}
	\tau^{m/2}
	s^{\left[ \sigma - n \left(\frac{1}{q}-\frac{1}{q'}\right) \right](m+1)}
	\n{\psi}_{L^q},
	\end{multline}
	where $1<q\le2$. Now, choosing $n \left(\frac{1}{q}-\frac{1}{q'} \right) \le \sigma$, putting together \eqref{Zpd}, \eqref{Zinter} and \eqref{Zh} with an application of Lemma \ref{lem:brenner}, we finally obtain the $L^q-L^{q'}$ estimate for $\partial_t W_1(s,t,D_x)$.

	%
	\medskip
	\textbf{Estimates for $\partial_t W_2(s,t,D_x)$.}
	For the intermediate and high frequencies cases, proceeding as above we straightforwardly obtain, under the constrains
	$\sigma \le \mu + n (\frac{1}{q}-\frac{1}{q'})$ and $n \left(\frac{1}{q}-\frac{1}{q'} \right) \le \sigma$ respectively, that
	\begin{equation}\label{W2chi_j}
	\n{F^{-1}_{\xi\to x} \left( X_j(\phi(s)\xi) |\xi|^{-\sigma} \partial_t W_2(s,t,|\xi|) \widehat{\psi} \right)}_{L^{q'}}
	\lesssim
	\tau^{m/2} s^{\left[\sigma-n \left( \frac{1}{q} - \frac{1}{q'} \right) \right](m+1)}
	\n{\psi}_{L^q}.
	\end{equation}
	for $j\in\{1,2\}$ and $1<q\le 2 \le q'<\infty$.
	
	At low frequencies, by computations similar to that for $\partial_tW_1(s,t,D_x)$, we obtain
	\begin{multline*}
	\n{F^{-1}_{\xi\to x} \left( X_0(\phi(t)\xi) |\xi|^{-\sigma} \partial_t W_2(s,t,|\xi|) \widehat{\psi} \right)}_{L^{q'}}
	\\
	\lesssim
	t^{(n/q'-n+\sigma)(m+1)}
	\n{T_0 * \F^{-1}_{\eta\to y}\left(\widehat{\psi}(\eta/\phi(t))\right)}_{L^{q'}}
	\end{multline*}
	where this time
	\begin{align*}
	T_0 :=&\, \F^{-1}_{\eta\to y} \left(
	X_0(\eta) |\eta|^{-\sigma} e^{-i[1+1/\tau^{m+1}]|\eta|}
	{\Phi}_0(\mu;\tau;|\eta|)
	\right)
	\\
	{\Phi}_0(\mu;\tau;|\eta|)
	:=&\,
	\Phi(\mu,2\mu;2i|\eta|/\tau^{m+1})
	\\
	&\times
	[\Phi(1-\mu,1-2\mu;2i|\eta|)-i(m+1)|\eta|\Phi(1-\mu,2(1-\mu);2i|\eta|)]
	\\
	=&\,
	[1+\tau^{-(m+1)}O(|\eta|)][1+O(|\eta|)],
	\end{align*}
	and hence again $|{\Phi}_0(\mu;\tau;|\eta|)|\lesssim 1$ if $|\eta|\le 3/4$. For any $\l>0$, we get
	\begin{align*}
	\meas\{ \eta\in\R^n \colon |\F_{y \to \eta}(T_0)| \ge \l\}
	\le
	\meas\{ \eta\in\R^n \colon \text{$|\eta| \le 3/4$ and $|\eta|^{-\sigma} \gtrsim \l$} \}
	\lesssim \l^{-b},
	\end{align*}
	where $1<b<\infty$ if $\sigma \le 0$ and $1 <b \le \frac{n}{\sigma}$ if $\sigma>0$.
	Another application of Lemma \ref{lem:hormander} tell us that $T_0 \in L^{q'}_q$ for $1<q\le 2 \le q'<\infty$ with the condition on $\sigma$ given by $\sigma \le n (\frac{1}{q}-\frac{1}{q'})$. Finally, similarly as in the case of $\partial_t W_1(s,t,D_x)$ we conclude that \eqref{W2chi_j} holds true also for $j=0$.
	
	The proof of estimates (iv) in Theorem\til\ref{thm:estW} is thus reached combining \eqref{W2chi_j} for $j\in\{0,1,2\}$, and putting together all the constrains on the range of $\sigma$, we are forced to choose $\sigma = n (\frac{1}{q}-\frac{1}{q'})$.

\section*{Acknowledgement}

The first author is supported by NSF of Zhejiang Province (LY18A010008) and NSFC (11771194). The second author is member of the Gruppo Nazionale per L'Analisi Matematica, la
Probabilit\`a e le loro Applicazioni (GNAMPA) of the Istituto Nazionale di
Alta Matematica (INdAM) and partially supported by the GNAMPA project \lq\lq Equazioni di tipo dispersivo: teoria e metodi'' and by \lq\lq Progetti per Avvio alla Ricerca di Tipo 1 -- Sapienza Universit\`a di Roma''.


\bibliographystyle{plain}

\begin{thebibliography}{99}
	
	\bibitem{AS}
	M.Abramovitz, and I.A.Stegun, Handbook of Mathematical Functions, 10th printing, National Bureau of Standards (1972).
	
	\bibitem{Age91}
	R.Agemi, Blow-up of solutions to nonlinear wave equations in two space dimensions, Manuscripta Math., 73 (1991), 153-162.
	
	\bibitem{BG1}
	J.Barros-Neto and I.M.Gelfand,
	Fundamental solutions for the Tricomi operator, Duke Math. J. 98 (1999), no. 3, 465-483.
	
	\bibitem{BG2}
	J.Barros-Neto and I.M.Gelfand,
	Fundamental solutions for the Tricomi operator. II, Duke Math. J. 111 (2002), no. 3, 561-584.
	
	
	\bibitem{BG3}
	J.Barros-Neto and I.M.Gelfand, Fundamental solutions of the Tricomi operator III, Duke Math. J., 128 (2005), no. 1, 119-140.		
	
	\bibitem{Bre75}
	P.Brenner, On $L^p-L^{p'}$ estimates for the wave-equation, Mathematische Zeitschrift 145.3 (1975), 251-254.
	
	
	\bibitem{CLP}
	W.Chen, S.Lucente and A.Palmieri, Nonexistence of global solutions for generalized Tricomi equations with combined nonlinearity (2020). arXiv preprint arXiv:2011.05906.
	
	\bibitem{DD01}
	P.D'Ancona, and A.Di Giuseppe, Global existence with large data for a nonlinear weakly hyperbolic equation, Mathematische Nachrichten 231.1 (2001), 5-23.
	
	\bibitem{ER18}
	M.R.Ebert, and M.Reissig, Methods for partial differential equations, Qualitative Properties of Solutions, Phase Space Analysis, Semilinear Models (2018).
	
	\bibitem{Fra}
	F.Frankl, On the problems of Chaplygin for mixed sub- and supersonic flows,
	Bull. Acad. Sci. URSS. Ser. Math., 9 (1945), 121-143.
	
	\bibitem{Gla83}
	R.T.Glassey, MathReview to \lq\lq Global behavior of solutions to nonlinear wave equations in three space dimensions'' of Sideris, Comm. Partial Differential Equations (1983).
	
	\bibitem{HH}
	M.Hamouda, and  M.A.Hamza, Blow-up and lifespan estimate for the generalized Tricomi equation with mixed nonlinearities (2020). arXiv preprint arXiv:2011.04895.
	
	\bibitem{He0}
	D.Y.He, Critical Exponents for Semilinear Tricomi-type Equations, Doctoral dissertation, Georg-August-Universit\"at G\"ottingen (2016).
	
	\bibitem{He1}
	D.Y.He, I.Witt and H.C.Yin, On semilinear Tricomi equations with critical exponents or in two space dimensions, J. Differential Equations, 263 (2017), no. 12, 8102-8137.
	
	\bibitem{He2}
	D.Y.He, I.Witt and H.C.Yin, On the global solution problem for semilinear generalized Tricomi equations, I, Calc. Var. Partial Differential Equations, 56 (2017), no. 2, Art. 21, 24 pp.
	
	\bibitem{He3}
	D.Y.He, I.Witt and H.C.Yin,
	On the global solution problem of semilinear generalized Tricomi equations, II, arXiv:1611.07606.
	
	\bibitem{He4}
	D.Y.He, I.Witt and H.C.Yin,
	On semilinear Tricomi equations in one space dimension, arXiv:1810.12748.
	
	\bibitem{HT95}
	K.Hidano and K.Tsutaya, Global existence and asymptotic behavior of solutions for nonlinear wave equations, Indiana Univ. Math. J., 44 (1995), 1273-1305.
	
	\bibitem{HWY12}
	K.Hidano, C.Wang and K.Yokoyama, The Glassey conjecture with radially symmetric data, J. Math. Pures Appl., 98(9) (2012), 518-541.
	
	\bibitem{Hor60}
	L.H\"ormander, Estimates for translation invariant operators in $L^p$ spaces, Acta Mathematica 104.1-2 (1960), 93-140.


	\bibitem{ILT}{M.Ikeda, J.Y.Lin and Z.H.Tu},
	Small data blow-up for the weakly coupled system of the generalized Tricomi equations with multiple propagation speed, arXiv:2010.03156.
	
	
	\bibitem{ISWa}{M.Ikeda, M.Sobajima and K. Wakasa},
	Blow-up phenomena of semilinear wave equations and their weakly coupled systems, J. Differential Equations, 267 (2019), 5165-5201.

	\bibitem{Inui62}
	T.Inui, Special Function, Iwanami Shoten, Tokyo, (1962).
	
	\bibitem{John81}
	F.John, Blow-up for quasilinear wave equations in three space dimensions, Comm. Pure Appl. Math., 34 (1981), 29-51.
	
	\bibitem{John85}
	F.John, Non-existence of global solutions of $\square u = \frac{\partial}{\partial t} F(u_t)$ in two and three space dimensions, Rend. Circ. Mat. Palermo (2) Suppl., 8 (1985), 229-249.
	
	\bibitem{LST20}
	N.A.Lai, N.M.Schiavone, and H.Takamura, Heat-like and wave-like lifespan estimates for solutions of semilinear damped wave equations via a Kato's type lemma, J. Differential Equations, 269 (2020), 11575-11620.
	
	\bibitem{LT}{N.A.Lai and H.Takamura},
	Blow-up for semilinear damped wave equations with
	sub-Strauss exponent in the scattering case, Nonlinear Anal., 168 (2018), 222-237.
	
	\bibitem{LT2}
	{N.A.Lai and H.Takamura}, Nonexistence of global solutions of nonlinear wave equations with weak time-dependent damping related to Glassey's conjecture. Differ Integral Equ, 32 (2019), 37-48.
	
	
	
	\bibitem{LT1}{N.A.Lai and Z.H.Tu},
	Strauss exponent for semilinear wave equations with scattering space dependent damping, J. Math. Anal. Appl. 489 (2020) 124189.
	
	
	\bibitem{LinTu}{J.Y.Lin and Z.H.Tu}, Lifespan of semilinear generalized Tricomi equation with Strauss type exponent, arXiv:1903.11351.
	
	
	\bibitem{LP}{S.Lucente, and A.Palmieri },
	A blow-up result for a generalized Tricomi equation with nonlinearity of derivative type (2020). arXiv prepint arXiv:2007.12895.

	\bibitem{Mas83}
	K.Masuda, Blow-up solutions for quasi-linear wave equations in two space dimensions, Lect. Notes Num. Appl. Anal., 6 (1983), 87-91.
	
	\bibitem{Ram87}
	M.A.Rammaha, Finite-time blow-up for nonlinear wave equations in high dimensions, Comm. Partial Differential Equations 12(6) (1987), 677-700.
	
	\bibitem{Rei97a}
	M.Reissig, On $L_p - L_q$ estimates for solutions of a special weakly hyperbolic equation. Nonlinear Evolution Equations and Infinite-Dimensional Dynamical Systems (1997), 153-164.
	
	\bibitem{Rei97}
	M.Reissig, Weakly hyperbolic equations with time degeneracy in Sobolev spaces. Abstract and Applied Analysis. Vol. 2. No. 3-4, Hindawi (1997), 239-256.
	
	\bibitem{RWY1}
	Z.-P.Ruan, I.Witt, and H.-C.Yin. The existence and singularity structures of low regularity solutions to higher order degenerate hyperbolic equations. J. Differential Equations 256 (2014), 407–460.
	
	\bibitem{RWY2}
	Z.-P.Ruan, I. Witt, and H.-C. Yin. On the existence and cusp singularity of solutions to semilinear generalized Tricomi equations with discontinuous initial data. Communications in Contemporary Mathematics 17.03 (2015): 1450028.
	
	\bibitem{RWY3}
	Z.-P.Ruan, I.Witt, and H.-C.Yin. On the existence of low regularity solutions to semilinear generalized Tricomi equations in mixed type domains. J. Differential Equations 259 (2015), 7406–7462.
	
	\bibitem{RWY4}
	Z.-P.Ruan, I.Witt, and H.-C.Yin. Minimal regularity solutions of semilinear generalized Tricomi equations. Pacific Journal of Mathematics 296 (1) (2018), 181-226.
		
	\bibitem{RS11}
	T.Runst and W.Sickel, Sobolev spaces of fractional order, Nemytskij operators, and nonlinear partial differential equations, Vol. 3. Walter de Gruyter, (2011).
	
	\bibitem{Sch86}
	J.Schaeffer, Finite-time blow up for $u_{tt}-\Delta u = H(u_r,u_t)$ in two space dimensions, Comm. Partial Differential Equations, 11 (5) (1986), 513-543.
	
	\bibitem{Sid83}
	T.C.Sideris, Global behavior of solutions to nonlinear wave equations in three space dimensions, Comm. Partial Differential Equations, 8 (12) (1983), 1291-1323.
	
	\bibitem{TT80}
	K.Taniguchi, Y.Tozaki, A hyperbolic equation with double characteristics which has a solution with branching singularities, Math. Japon 25 (1980), 279–300.
	
	\bibitem{Tri23}
	F.Tricomi. Sulle equazioni lineari alle derivate parziali di secondo ordine, di tipo misto. Atti Accad. Naz. Lincei Mem. Cl. Fis. Mat. Nat. 5 (1923), 134–247.
	
	\bibitem{Tzv98}
	N.Tzvetkov, Existence of global solutions to nonlinear massless Dirac system and wave equation with small data, Tsukuba J. Math., 22 (1998), 193-211.
	
	
	\bibitem{Wang}
	C.B.Wang, The Glassey conjecture on asymptotically flat manifolds. Trans Amer Math Soc, 2015, 367: 7429-7451
	
	\bibitem{Yag1}{K.Yagdjian},
	A note on the fundamental solution for the Tricomi-type equation in the hyperbolic
	domain, J. Differential Equations 206 (2004), 227-252.
	
	\bibitem{Yag2}{K.Yagdjian},
	Global existence for the $n$-dimensional semilinear Tricomi-type equations, Comm.
	Partial Differential Equations, 31(2006), no.4-6, 907-944.
	
	\bibitem{Yag3}{K.Yagdjian}, The self-similar solutions of the one-dimensional semilinear Tricomi-type equations,
	J. Differential Equations, 236 (2007), no.1, 82-115.
	
	\bibitem{Yag4}{K.Yagdjian}, The self-similar solutions of the Tricomi-type equations, Z. Angew. Math. Phys.,
	58 (2007), no.4, 612-645.
	
	\bibitem{Yag5}{K.Yagdjian}, Self-similar solutions of semilinear wave equation with variable speed of propagation,
	J. Math. Anal. Appl., 336 (2007), no.2, 1259-1286.
	
	\bibitem{Zhou01}
	Y.Zhou, Blow-up of solutions to the Cauchy problem for nonlinear wave equations, Chin. Ann. Math., 22B (3) (2001), 275-280.
	
\end{thebibliography}

\end{document}